\documentclass[twoside,11pt,titlepage]{amsart}

\usepackage{amsthm} 
\usepackage{amssymb}	
\usepackage{amsmath} 
\usepackage{graphicx} 
\usepackage{multicol} 
\usepackage[marginratio=1:1,height=584pt,width=360pt,tmargin=117pt]{geometry}
\usepackage{hyperref}
\usepackage{pst-node}
\usepackage{tikz-cd}
\usepackage{tikz}
\usetikzlibrary{calc}
\usepackage[mathscr]{euscript}
\usepackage[numbers]{natbib}
\usepackage[toc,page]{appendix}
\usepackage{dsfont}
\usepackage{faktor}
\usepackage{multirow}
\usepackage{url,times,wasysym,geometry,indentfirst,rotating,tikz}

\renewcommand{\Vec}{\text{Vec}}
\DeclareMathOperator{\vgo}{\text{Vec}_G}

\DeclareMathOperator{\Endo}{Endo}

\DeclareMathOperator{\Hom}{Hom}

\DeclareMathOperator{\Mod}{Mod}
\DeclareMathOperator{\Modc}{Modc}
\DeclareMathOperator{\Bimodc}{Bimodc}
\DeclareMathOperator{\Span}{Span}

\DeclareMathOperator{\VecZp}{\Vec_{\mathbb{Z}_p}}
\DeclareMathOperator{\Haag}{\text{Haag}}

\newcommand{\mcA}{\mathcal{A}}
\newcommand{\mcB}{\mathcal{B}}
\newcommand{\mcC}{\mathcal{C}}
\newcommand{\mcD}{\mathcal{D}}

\newcommand{\mcF}{\mathcal{F}}

\newcommand{\mcH}{\mathcal{H}}

\newcommand{\mcM}{\mathcal{M}}

\newcommand{\mcZ}{\mathcal{Z}}

\newcommand{\mbbC}{\mathbb{C}}

\newcommand{\mbbR}{\mathbb{R}}

\newcommand{\mbbZ}{\mathbb{Z}}

\def\Aut#1{\text{Aut}_{\otimes}^{br}(#1)} 
\def\uAut#1{\underline{\text{Aut}_{\otimes}^{br}(#1)}}
\def\uuAut#1{\underline{\underline{\text{Aut}_{\otimes}^{br}(#1)}}}
\def\cross#1#2{{#1}_{#2}^{\times}}
\def\gauge#1#2#3{{#1}_{#2}^{\times,#3}}
\def\Pic#1{\text{Pic}(#1)}
\def\uPic#1{\underline{\text{Pic}(#1)}}
\def\uuPic#1{\underline{\underline{\text{Pic}(#1)}}}
\def\u#1{\underline{#1}}
\def\uu#1{\underline{\underline{#1}}}

\def\SU{\text{SU}}
\def\unit{\mathbf{1}}

\def\Fib{\text{Fib}}

\def\Fib{\text{Fib}}

\def\scale{0.7}
\def\golden{\phi}
\def\Un{U}
\def\DFib{\text{DFib}}
\def\L{\text{L}}

\let\baraccent=\= 
\renewcommand{\=}[1]{\stackrel{#1}{=}} 
\renewcommand{\Vec}{\text{Vec}}
\title{{\LARGE\bf
\textsc{On Generalized Symmetries and Structure of Modular Categories
}}}
\newcommand\Author{Shawn X. Cui, Modjtaba Shokrian Zini \& Zhenghan Wang}

\makeatletter
\newcommand{\addresseshere}{%
  \enddoc@text\let\enddoc@text\relax
}
\makeatother
\makeatletter 
\g@addto@macro{\endabstract}{\@setabstract}
\newcommand{\authorfootnotes}{\renewcommand\thefootnote{\@fnsymbol\c@footnote}}%
\makeatother
\makeatletter
\renewcommand{\maketitle} 
{ \begingroup \vskip 10pt \begin{center} \large {\bf \@title}
	\vskip 10pt \large \@author \hskip 20pt \@date \end{center}
  \vskip 10pt \endgroup \setcounter{footnote}{0} }
\makeatother 
\makeatletter
\let\Title\@title
\def\ps@mystyle{%
      \let\@oddfoot\@empty\let\@evenfoot\@empty
      \def\@evenhead{\makebox[0pt][l]{\thepage}\hfill\Author\hfill}%
      \def\@oddhead{\hfill\Title\hfill\makebox[0pt][l]{\thepage}}%
      \let\@mkboth\markboth}
\makeatother

\newtheorem{prop}{Proposition}
\newtheorem{thm}{Theorem}[section]

\newtheorem{opq}[thm]{Open Question}

\newtheorem{conjecture}[thm]{Conjecture}
\theoremstyle{definition}
\newtheorem{definition}{Definition}
\theoremstyle{remark}
\newtheorem{remark}{Remark}

\begin{document}
\begin{center}
  \LARGE 
   \maketitle \par \bigskip

  \normalsize
  \authorfootnotes
  Shawn X. Cui \footnote{xingshan@stanford.edu; xingshan@vt.edu}\textsuperscript{,\hyperref[1.]{1},\hyperref[2.]{2}}, Modjtaba Shokrian Zini \footnote{shokrian@math.ucsb.edu}\textsuperscript{,\hyperref[3.]{3}}, Zhenghan Wang \footnote{zhenghwa@microsoft.com;
  zhenghwa@math.ucsb.edu}\textsuperscript{,\hyperref[4.]{4}} \par 
  \bigskip
\end{center}

\begin{center}
\begin{abstract}
Pursuing a generalization of group symmetries of modular categories to category symmetries in topological phases of matter, we study linear Hopf monads. The main goal is a generalization of extension and gauging group symmetries to category symmetries of modular categories, which include also categorical Hopf algebras as special cases.  As an application, we propose an analogue of the classification of finite simple groups to modular categories, where we define simple modular categories as the prime ones without any nontrivial normal algebras.   
\end{abstract}
\end{center}

\tableofcontents

\section{Introduction}\label{1}

Group symmetries are central in mathematics and physics.  But group symmetry alone is not adequate to capture phenomena such as symmetry fractionalization, symmetry defects, and gauging in topological phases of matter, where categorical group symmetries are used \cite{BBCW14,cui2016gauging}.  In this paper, we propose a further generalization from categorical groups to fusion categories using linear Hopf monads, and investigate the basic properties for such generalized symmetries.  The main goal is a generalization of extension and gauging group symmetries to category symmetries of modular categories, which include also categorical Hopf algebras as special cases.  As an application, we propose an analogue of the classification of finite simple groups to modular categories, where we define simple modular categories as the prime ones without any nontrivial normal algebras.

Exemplary examples of group symmetries in the mathematical formulation of quantum field theories include $\textrm{U}(1), \SU(2), \SU(3)$ in the standard model.  Quantum groups, loosely defined as Hopf algebras, appeared in the study of inverse scattering problems in mathematical physics and are instrumental in our understanding topological phases of matter, which lie outside the Landau paradigm of phases of matter based on group symmetries.   Topological phases of matter in two spatial dimensions are modeled by anyon models, mathematically unitary modular categories.  Finite group symmetries of anyon models have been extensively studied recently revealing a deep interplay between symmetry and topological order through symmetry fractionalization, defects, and gauging \cite{BBCW14,cui2016gauging}.  How to extend symmetries beyond finite groups arise as a natural problem.  Our generalization to linear Hopf monads as symmetries of anyon models and unitary fusion categories can be equivalently formulated as fusion category symmetries of modular categories.

A linear Hopf monad is a far-reaching generalization of a Hopf algebra \cite{bruguieres2007hopf,bruguieres2011hopf}, which is closely related to adjunctions of tensor functors.  A Hopf monad $T: \mcC \rightarrow \mcC$ of a fusion category $\mcC$ is essentially the same as a tensor functor $D: \mcC^T \rightarrow \mcC$, where $\mcC^T$, assuming $T$ is semisimple, is the \textit{fusion} category of Eilenberg-Moore category of $T$-modules. Two large classes of Hopf monads are those from group actions $G$ on $\mcC$ and categorical Hopf algebras $\mcH$ in $\mcC$.  

One of the motivations for this work is the possibility of generalizing the doubled Haagerup category by gauging some Hopf monad symmetries on abelian modular categories\footnote{The third author would like to thank C. Delaney, T. Gannon, and J. Tener for collaborating on this project.}. It is unclear at the moment if this approach is a viable one.

The contents of the paper are as follows.  In \hyperref[3]{section 3}, we first recall the definition of Hopf monad \cite{bruguieres2011hopf} that does not follow the standard algebra-coalgebra one \cite{bruguieres2007hopf} using the so-called fusion functors.  In \hyperref[4]{section 4}, we skeletalize the categorical definition of Hopf monads, which are useful for concrete calculation. In \hyperref[5]{section 5}, we attempt to characterize Hopf monads on $\VecZp$.  In \hyperref[6]{section 6}, we outline the theory of extension, gauging, and generalized symmetry defects.  Finally, we propose a structure theory for modular categories.

\section{Background material and a list of notations}\label{2}

Fusion categories and braided fusion categories in this paper are assumed to be unitary over $\mbbC$. Therefore, all fusion categories are spherical. All Hopf monads (HM) are $\mbbC$-linear and semi-simple ensuring their Eilenberg-Moore module categories to be fusion \cite[Remark 6.2]{bruguieres2007hopf}. We follow the book \cite{etingof2016tensor} for the various notions of categories such as abelian, monoidal, braided, unitary, and fusion, and similarly for notions of functors with the additional constraint that all functors considered will be linear. Also, there will be an exception: A monoidal functor $F$ has maps
$$F_2(X,Y): F(X) \otimes F(Y) \to F(X \otimes Y),\ \ 
F_0: \unit \to F(\unit)$$
which are \textit{not} necessarily isomorphisms. If they are, the functor will be called \textit{strong} monoidal. As a result, a tensor functor is (linear) strong monoidal functor which is exact and faithful. We shall also assume well-known facts about module categories (see e.g. \cite[Chap. 7]{etingof2016tensor} or \cite{ostrik2003module}).

\textit{Notations:}
\begin{enumerate}

\item $\mcC, \mcD, \mcF:$ fusion categories

\item $\mcA,\mcH:$ algebras, usually condensable and Hopf, respectively

\item $\mcB:$ unitary braided fusion category

\item $\mcB_\mcA$: the condensed fusion category from condensing an algebra $\mcA$ of $\mcB$

\item $\Endo(\mcC):$ the monoidal category of endo-functors of $\mcC$

\item $\Modc(\mcC):$ the category of module categories of $\mcC$

\item $\Bimodc(\mcC):$ the linear monoidal category of bimodule categories of $\mcC$

\item $\Modc(\mcB):$ the linear monoidal category of module categories of $\mcB$ regarded as bimodule categories using the braiding, not the same as $\Bimodc(\mcB)$

\item $\Aut{\mcB}$, $\uAut{\mcB}$, $\uuAut{\mcB}$: The 1-group, 2-group, 3-group of braided equivalences of $\mcB$.

\item Similarly, $\Pic{\mcB}, \uPic{\mcB}, \uuPic{\mcB}$: The 1-group, 2-group, 3-group of invertible module categories of $\mcB$.

\item $\mcC^T$: the fusion category of $T$-modules of a Hopf monad $T\in \Endo(\mcC)$, with objects $(M,r), M \in \mcC,$ and $r:T(M) \to M$ the $T$-action.

\item $\cross{\mcB}{T}$: $T$-crossed braided category extended over $\mcB$

\item $\gauge{\mcB}{T}{T}$: the gauging of $\mcB$

\item $\u{\rho}$: categorical group homomorphism

\item $\uu{\rho}$: categorical 2-group homomorphism

\item $\mcZ(\mcC)$ and $\mathcal{D}(\mcC)$ are the Drinfeld center of $\mcC$

\item $\vgo$ is the fusion category of $G$-graded vector spaces, and $\mathcal{D}(G)$ also denotes the Drinfeld center $\mcZ(\vgo)$ of $\vgo$. 

\end{enumerate}

\section{Linear Hopf monads}\label{3}

In this section, we define Hopf monad following Remark 2.7 in \cite{bruguieres2011hopf}.  We collect the many different occurrences of HM in the literature and focus on condensation.  

\subsection{Hopf monads}\label{3.1}

\begin{definition}\label{dfn1}
A left Hopf monad on a fusion category $\mcC$ with associator $a_{X,Y,Z}$ is a quadruple $(T,H,\eta, T_0)$  where $T$ is an endofunctor of $\mcC$, $H = \{H_{X,Y}\ | \ X,Y \in \mcC\}$ is a natural isomorphism from $T\left(- \otimes T(-)\right)$ to $T(-) \otimes T(-)$, $\eta = \{\eta_X\ | \ X \in \mcC\}$ is a natural transformation from $Id_{\mcC}$ to $T$, and $T_0\colon T(\unit) \longrightarrow \unit$ is a morphism, such that the following equations hold.
\begin{itemize}
    \item Compatibility condition between $H$ and $\eta,\, T_0$:
     \begin{equation}
     \label{equ:HM_triangle_axioms}
     \begin{split}
        H_{X,Y} \eta_{X \otimes T(Y)} = \eta_{X} \otimes id_{T(Y)}, \qquad & T_{0} \eta_{\unit} = id_{\unit},\\
        (id_{T(X)} \otimes T_0)H_{X,\unit} = T(id_{X} \otimes T_0),    \qquad & (T_0 \otimes id_{T(X)} )H_{\unit, X}T(\eta_{X}) = id_{T(X)}.
     \end{split}
    \end{equation}
    \item Heptagon Equation\footnote{If $\mcC$ is strict, then two of the maps appearing in the Heptagon Equation become the identity map, and hence the equation is reduced to a \lq Pentagon' Equation, which was first introduced in \cite{bruguieres2011hopf} and is not to be confused with the Pentagon Equation arising the associator of a monoidal category or the heptagon equations in the $G$-crossed braiding.},
    $$(id_{T(X)} \otimes H_{Y,Z}) H_{X, Y \otimes Z}=$$
    \begin{align}
    \label{equ:HM_heptagon_axiom}
       a_{T(X),T(Y),T(Z)} (H_{X,Y} \otimes T(Z)) H_{X \otimes T(Y), Z} a_{X,T(Y),T(Z)}^{-1} T(id_X \otimes H_{Y,Z})
    \end{align}
\end{itemize}
\end{definition}
From the fusion operator one can derive the multiplication $\mu : T^2 \to T$ as follows:
$$(T_0 \otimes id_{T(X)})H_{\unit,X}=\mu_X.$$
Further the comonoidal structure of $T$ is given by
$$T_2(X,Y)=H_{X,Y}T(id_X \otimes \eta_Y): T(X\otimes Y) \to T(X) \otimes T(Y).$$
We shall call the operator $H$ the left fusion operator $H^l_{X,Y}$, and the \textit{right} fusion operator to be $H^r_{X,Y}$ which operates on $T(T(-)\otimes -)$ instead. They satisfy similar equations as illustrated in \cite{bruguieres2011hopf}. A right Hopf monad has a definition similar to above using $H^r_{X,Y}$. We shall always assume a HM $T$ to be a left and right HM.
\begin{definition}\label{dfn2}
A module $(M,r)$ of $T$ is an object $M$ of $\mcC$ with $r:T(M)\to M$ a morphism of $\mcC$ such that
$$rT(r)=r\mu_M \ \ \ , \ \ \ r\eta_M=id_M $$
The category of $T$-modules is denoted by $\mcC^T$ and its morphisms are $\mcC$ morphisms preserving the $T$-module structure.
\end{definition} 
The HMs we shall work with are always assumed to be semisimple linear and as pointed out in introduction, for such HMs, the module category is fusion \cite{bruguieres2007hopf}.

Two important functors related to $T$ are the forgetful functor $F:\mcC^T \to \mcC$ which forgets the module structure, and its left adjoint $G : M \to (T(M),\mu_M)$ which gives $T=FG$. This illustrates why every Hopf monad is an adjunction \cite[Remark 3.16]{bruguieres2007hopf}.

A definition that will be needed to tie the condensable algebra topic to HMs is the following:
\begin{definition}\label{dfn3}
A Hopf comonad $T:\mcC \to \mcC$ is a bicomonad, i.e. a monoidal comonad with a natural morphism called counit $\epsilon: T \to 1_\mcC$, a comultiplication $\Delta: T \to T^2$, a monoidal morphism $T_2: T(x)\otimes T(Y) \to T(X \otimes Y)$ and a unit morphism of objects $T_0: \unit \to T(\unit)$, plus a left and right fusion operator satisfying the same axioms of HMs.

Just like HMs, our Hopf comonads will always be assumed to be semi-simple linear.
\end{definition}
One way to picture a Hopf comonad is to imagine all structural morphisms and axioms of HMs with arrows reversed. A Hopf comonad can also be recovered by taking the \textit{right} adjoint of a tensor functor. Further, the \textit{co}modules of a Hopf comonad, denoted by $\mcC^{co-T}$, form a fusion category. 

It can be easily observed that HMs like $T_\mcH(-)=\mcH \otimes -$ are also Hopf comonads. Therefore, it is useful to have the following definition:
\begin{definition}\label{dfn4}
A HM $T$ that is also a Hopf comonad is called a self-dual HM.
\end{definition}
For a self-dual HM $\mcC^T$ and $\mcC^{co-T}$ will denote the module and comodule category, respectively.

\subsection{Examples of HMs}\label{3.2}
\subsubsection{HMs from group actions}\label{3.2.1}
Consider a $G-$symmetry of a finite group $G$ on $\mcC$ defined by a strong monoidal functor $T_{-}: \underline{G} \to \underline{\text{Aut}_\otimes(\mcC)}$. One can form the operator $T_G=\oplus_{g \in G} T_g$ where $T_g$ are the symmetries. This is indeed a HM \cite[Thm. 4.21]{bruguieres2011exact}. To see why, the counit $T_0$ and unit $\eta$ and the fusion operator $H$ are the analog of the same operators which shows that $\mathbb{C}[G]$ is a Hopf algebra. To define the fusion operator we can use our knowledge of the coproduct 
$$T_2 : \oplus T_g(X\otimes Y) \to \oplus T_g(X)\otimes T_g(Y),$$ which is derived from the the fact that $T_g$ are (strong monoidal) automorphisms, and the definition of the fusion operator in \cite[Intro.]{bruguieres2011hopf} as
$$H^l_{X,Y}=(id_{T(X)}\otimes \mu_Y)T_2(X,T(Y))$$
and similarly for the right $H^r_{X,Y}$. Straightforward calculation shows that they are inverse to each other. This proves that $T_G$ is a HM.

\subsubsection{HMs from categorical Hopf algebras}\label{3.2.2}
The \textit{action} of a Hopf algebra was used as a venue to generalize its notion to all monoidal categories. This provides the most basic general class of HMs \cite{bruguieres2007hopf}; a Hopf algebra $\mcH$ inside $\mcC$ gives a HM $T_\mcH(-)=\mcH \otimes -$. These are Hopf algebra symmetries. Notice Hopf algebra needs a braiding for its definition, therefore a braiding structure for the relevant category is always assumed.

\subsubsection{HMs from tensor functors}\label{3.2.3}
Tensor functors are extra structures on functors between fusion categories. A tensor functor $F: \mcD \to \mcC$ between fusion categories has adjoints (see e.g. \cite[1.3]{bruguieres2011exact}) and its left adjoint $G$ provides a HM $T=FG$ and an equivalence $\mcD \cong \mcC^T$. Conversely, any HM $T$ has a (forgetful) tensor functor $U:\mcC^T \to \mcC$ with the left adjoint being $X \to (T(X),\mu_X)$. 

\subsubsection{HMs from condensation}\label{3.2.4}
We first define a condensable algebra. To do so, recall an algebra $\mcA$ of a braided fusion category $\mcB$ with multiplication $m:\mcA \otimes \mcA \to \mcA$ is
\begin{itemize}
    \item commutative if $m \circ c_{\mcA,\mcA} =m $ where $c_{\mcA,\mcA}$ is the braiding,
    \item connected if $\Hom(\unit,\mcA)\cong \mbbC$,
    \item separable if $m$ admits a splitting $\zeta: \mcA \to \mcA \otimes \mcA$, a morphism of $(\mcA,\mcA)$-bimodules.
\end{itemize} 
\begin{definition}\label{dfn5}
An algebra $\mcA$ of a braided fusion category $\mcB$ is condensable if it is a connected etale (commutative and separable) algebra.

A condensable algebra $\mcA$ (physically) will also be referred to as normal (mathematically).
\end{definition}

For the HM $T_\mcH(-)=\mcH \otimes -$ given by a Hopf algebra, the forgetful functor is the tensor functor and we consider its left adjoint. The case of condensable algebras (which are not Hopf algebras in general) is opposite and we refer to \cite[Def. 3.11 and 3.12]{cong2016topological} for the details. The condensation functor $D_\mcA: \mcB \rightarrow \mcB_\mcA$ is the tensor functor, and its \textit{right} adjoint is the forgetful functor $E_\mcA: \mcB_\mcA \rightarrow \mcB$. The composition $T_\mcA=D_\mcA \circ E_\mcA: \mcB_\mcA\rightarrow \mcB_\mcA$ is a Hopf \textit{co}monad and the equivalence $\mcB \cong \mcB_\mcA^{co-T_\mcA}$ holds.

A simple object $X$ of $\mcB_\mcA$ is deconfined if all the simple objects of $E_\mcA(X)$ all have the same topological twist.  Otherwise it is confined and called a defect.

In order to derive the maps $D_\mcA$ and $E_\mcA$, the Frobenius reciprocity property is used:
\begin{align}
\Hom_{\mcB}(X,\mcA \otimes Y)=\Hom_{\mcB_\mcA}(X,Y).
\end{align}
Notice that objects of the two categories are actually the same although the simple objects are of course not. We should note that the above version of Frobenius reciprocity is true for the cases considered in this paper. Otherwise, instead of $\mcB_\mcA$ there should be $\widetilde{\mcB_\mcA}$ which is a pre-quotient category which needs to go through an idempotent completion process to give us $\mcB_\mcA$ (see \cite[Def. 3.11 and 3.12]{cong2016topological}). When there are no non-trivial splitting idempotents in $\widetilde{\mcB_\mcA}$, then $\widetilde{\mcB_\mcA}=\mcB_\mcA$. The Frobenius reciprocity provides enough constraints to derive completely the maps $D_\mcA,E_\mcA$ for the condensation examples in \hyperref[5.2]{5.2}.

\subsection{Double of HMs}\label{3.3}

The double of $T$ is defined using the coend
$$Z_T(x)= \bigoplus_i{}^*T(x_i)\otimes x\otimes x_i$$
as $D_T=Z_T \circ T$, where the summation is finite as we deal with UFCs (\cite[9.2]{bruguieres2007hopf}). The important example is $T=1_\mathcal{C}$ which gives the double of $\mathcal{C}$ , and in general (\cite[9.2]{bruguieres2007hopf}):
$$\mcC^{D_T}\cong\mathcal{Z}(\mcC^T),$$
as braided categories. This simple example also shows the power of Hopf monads, where the most simple HM, the identity map, can give rise to the double construction.

\subsection{Equivariantization of HM}\label{3.4}

The equivariantization $\mcC^G$ of the group action $G$ on $\mcC$ is known to be equivalent to $\mcC^{T_G}$ \cite[Thm. 4.21]{bruguieres2011exact}. Therefore, the natural generalization of equivariantization is the concept of $T$-modules. This also makes sense as these notions are supposed to categorify the notion of fixed point of an action. Therefore, the last step in the gauging process is the process of taking the $T-$modules.

\section{Skeletal definition of Hopf monads}\label{4}

In mathematics, many notions such as a fusion category are equivalence classes.  It is often convenient to work with different representatives of the same equivalence class.  In the case of a fusion category, two extreme convenient representatives are: a strict version with many objects, and a skeletal version with the least number of objects.  The strict version is convenient for general discussions and pictorial calculus, while the skeletal version is easier for concrete calculations.  In this section, we provide the skeletal definition of a HM, presented  as solutions of polynomial equations. As a warm-up, we start with the skeletal definition of tensor functors. 

\subsection{Diagrammatic Notations}\label{4.1}
Let $\mcC$ be a fusion category and  $\L(\mcC) = \{a,b,c, \cdots\}$ be a complete set of representatives, i.e., a set that contains a representative for each isomorphism class of simple objects of  $\mcC$. Let $N_{ab}^c$ be the fusion coefficients, 
\begin{align*}
    a \otimes b \simeq \oplus_{c \in \L(\mcC)} N_{ab}^c\, c.
\end{align*} 
For simplicity, we assume $\mcC$ to be multiplicity free, that is, $N_{ab}^c = 0,1$. A triple $(a,b,c)$ is called admissible if $N_{ab}^c = 1$. 

We introduce a diagrammatic way to represent the tensor product functor, an extra functor $T$, and their Cartesian product/compositions. We first assume $T$ to be an endofunctor of $\mcC$. Then $T$ is represented by an integer-valued matrix $(T_{ab})$:
\begin{align*}
    T(b) =  \oplus_{a \in \L(\mcC)} T_{ab}\, a, \qquad b \in \L(\mcC).
\end{align*}
Consider binary forests, i.e., a disjoint union of binary trees, with each edge possibly decorated with some number of dots. Every dot on the edge splits that edge into two. (If there are $n$ dots on an edge, then that edge is split into $n+1$ edges.) The forests have distinguished roots and  are aligned so that they grow in the upward direction. Thus roots appear at the bottom and leaves appear at the top.  Let $\mathcal{G}$ be a forest as described above. If $\mathcal{G}$ has $n$ leaves and $m$ roots, then it is interpreted as a functor $\mcC^n \to \mcC^m$ according to the following rules.
\begin{itemize}
    \item A \lq$Y$'-shape tree is interpreted as the tensor product functor $\otimes \colon \mcC^2 \to \mcC$; a verticle segment as the identity functor $Id\colon \mcC \to \mcC$; a verticle segment with a dot in it as $T\colon \mcC \to \mcC$.
    \item Horizontal juxtaposition of two forests corresponds to the Cartesian product of the respective functors. 
    \item Vertical concatenation of two forests corresponds to the composition of the respective functors. Explicitly, if the roots of $\mathcal{G}_1$ coincide with the leaves of $\mathcal{G}_2$, then the concatenated forest $\mathcal{G}_1 \# \mathcal{G}_2$ is interpreted as the composition of the functor corresponding to $\mathcal{G}_1$ and that corresponding to $\mathcal{G}_2$.  
\end{itemize}
See Figure \ref{fig:T_graph_examples} for examples of forests and their interpretations as functors. From now on, we will {only} study trees and we will use the same symbol for a tree and the functor it represents.

\begin{figure}
\centering
\begin{tikzpicture}[thick]
\begin{scope}
\draw (0,0) -- (-1,1);
\draw (0,0) -- (1,1);
\draw (0,0) -- (0,-1);

\draw (0,-1.5) node{I};
\end{scope}

\begin{scope}[xshift = 3cm]
\draw (0,1) -- (0,-1);
\draw (0,-1.5) node{II};
\end{scope}

\begin{scope}[xshift = 5cm]
\draw (0,1) -- (0,-1);
\fill (0,0) circle[radius = 2pt];
\draw (0,-1.5) node{III};
\end{scope}

\begin{scope}[xshift = 8cm]
\draw (0,0) -- (-1,1);
\draw (0,0) -- (1,1);
\draw (0,0) -- (0,-1);

\fill (0.5,0.5) circle[radius = 2pt];
\fill (0,-0.5) circle[radius = 2pt];
\draw (0,-1.5) node{IV};
\end{scope}

\begin{scope}[xshift = 12cm]
\draw (0,0) -- (-1,1);
\draw (0,0) -- (1,1);
\draw (0,0) -- (0,-1);

\fill (0.5,0.5) circle[radius = 2pt];
\fill (-0.5,0.5) circle[radius = 2pt];
\draw (0,-1.5) node{V};
\end{scope}

\end{tikzpicture}
\caption{(I) $\otimes$; (II) $Id_{\mcC}$; (III) $T$; (IV) $T \circ \otimes \circ (Id \times T)$; (V) $\otimes \circ (T \times T )$. }\label{fig:T_graph_examples}
\end{figure}
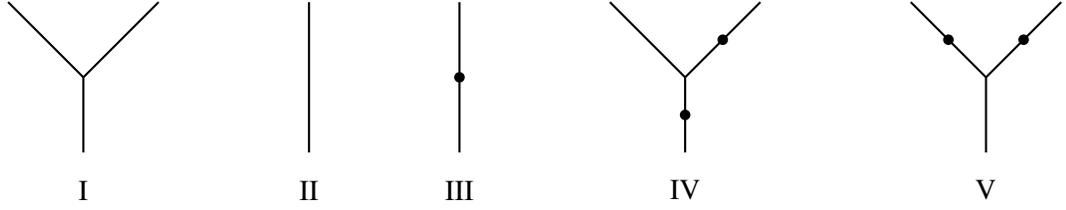

Let $\mathcal{G}$ be a binary tree with $n$ leaves representing a functor (also denoted by $\mathcal{G}$) from $\mcC^n$ to $\mcC$. We label each leaf $l_i$ with a simple object $a_i \in \L(\mcC)$. Then $\mathcal{G}(a_1, \cdots, a_n)$ is an object in $\mcC$ which may contain multiple copies of each simple object in $\L(\mcC)$. To specify a particular copy, we further label all other edges (including the one attached to the root) with simple objects in $\L(\mcC)$ and label all dots with positive integers so that the following two conditions hold:
\begin{enumerate}
    \item at each trivalent vertex, if the three edges (pointing towards to the northwestern direction, northeastern direction, southern direction) are respectively labelled by $a,b,c$, then $(a,b,c)$ is admissible.
    \item at each dot, if the edge above (respectively, below) the dot is labelled by $b$ (respectively, $a$) and if the dot is labelled by $\alpha$, then $1 \leq \alpha \leq T_{ab}$.
\end{enumerate}
See Figure \ref{fig:T_graph_labeling} for an illustration of labels of trees.
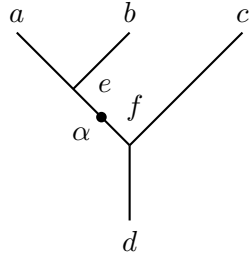
\begin{figure}
\centering
\begin{tikzpicture}[thick]
\begin{scope}[xshift = 0cm]
\draw (0,0) -- (-1.5,1.5)node[above]{$a$}node[pos = 0.25, anchor = north east]{$\alpha$}node[pos = 0.375, anchor = south west]{$e$}node[pos = 0.125, anchor = south west]{$f$};
\draw (0,0) -- (1.5,1.5)node[above]{$c$};
\draw (0,0) -- (0,-1)node[below]{$d$};
\draw (0,1.5)node[above]{$b$} -- (-0.75, 0.75);

\fill (-0.375,0.375) circle[radius = 2pt];
\end{scope}

\end{tikzpicture}
\caption{A labelling of a tree}\label{fig:T_graph_labeling}
\end{figure}

If the label for the edge attached to the root is $b$, then a labeling of $\mathcal{G}$ determines a specific copy of $b$ inside $\mathcal{G}(a_1,\cdots, a_n)$, or equivalently, a basis element in the space $\Hom\left(b,\mathcal{G}(a_1,\cdots, a_n)\right)$. By ranging all labels of internal edges, we enumerate all copies of $b$, or equivalently, a basis in $\Hom\left(b,\mathcal{G}(a_1,\cdots, a_n)\right)$. We will use these two interpretations interchangeably. 
\begin{remark}\label{rmk1}
Actually, the basis element in $\Hom\left(b,\mathcal{G}(a_1,\cdots, a_n)\right)$ corresponding to a labelling can be seen directly from the labelled tree. To achieve this, for each admissible $(a,b,c)$, fix a morphism $B_{c}^{ab} \in \Hom(c, a \otimes b)$ and for each $a,b \in \L(mcC)$, $1 \leq \alpha \leq T_{ab}$, fix an embedding $I_{a}^b(\alpha)\colon a \to T(b)$ so that the image of $I_{a}^b(\alpha)$ is the $\alpha$-th copy of $a$ inside $T(b)$.  Then, divide the tree into several layers such that within each layer the tree is a horizontal juxtaposition of at most three types of graphs: \lq$Y$'-shape graphs, verticle lines, and verticle lines with a dot. Each \lq$Y$'-shape graph with labels $(a,b,c)$ is interpreted as $B_{c}^{ab}$; each verticle line with a label $a$ is interpreted as $id_a \in \Hom(a,a)$; each verticle line with a dot whose labels are $(a,b;\alpha)$ with $b$ being the label of the edge above the dot is interpreted as $I_{a}^{b}(\alpha)$. Now, for the $i$-th layer, assume the morphisms we obtained by the above rules are $f_{i1}, \cdots, f_{in}$ and assume the part of the tree below the $i$-th layer represents the functor $F_i$, then the $i$-th layer is assigned the morphism $F_i(f_{i1}, \cdots, f_{in})$. Lastly, the morphism corresponding to the whole tree is obtained as the compositions of the morphisms assigned to each layer, from bottom to top. For instance, the labelled tree in Figure  \ref{fig:T_graph_labeling} represents the morphism:
\begin{equation}
\begin{tikzcd}
d \arrow[r, "B_{d}^{fc}"] & f \otimes c \arrow[rr, "I_f^e(\alpha) \otimes id_{c}"] && Te \otimes c \arrow[rr, "T(B_{e}^{ab}) \otimes id_c"] && T(a \otimes b) \otimes c.
\end{tikzcd}
\end{equation}
\end{remark}

To summarize, a binary tree $\mathcal{G}$ represents the functor $\mathcal{G}$; a binary tree with leaves labelled by $a_i$ and root labelled by $b$ represents the space $\Hom\left(b,\mathcal{G}(a_1,\cdots, a_n)\right)$; a binary tree with labels on each edge represents a particular copy of $b$ in  $\mathcal{G}(a_1,\cdots, a_n)$ or a particular basis element in  $\Hom\left(b,\mathcal{G}(a_1,\cdots, a_n)\right)$.

Let $\mathcal{G}_1, \mathcal{G}_2$ be two binary trees with the same labels $\{a_i\}$ on  leaves and $b$ on their root. If $\Phi$ is  a natural transformation from $\mathcal{G}_1$ to $\mathcal{G}_2$, then with these labels $\Phi$ can be written as a matrix $(\Phi_{b}^{\{a_i\}})_{\{c_i\},\{d_i\}}$, where $\{c_i\}$ ranges over all labels of internal edges of $\mathcal{G}_2$ and $\{d_i\}$ ranges those of $\mathcal{G}_1$. This is how we will represent natural transformations in the skeletal definition of tensor functors and Hopf monads below.   

As an example, the associator isomorphism for the tensor product and its matrix elements, the so called $F$-matrix, are represented as in Figure \ref{fig:T_graph_associator}. 
\begin{figure}
\centering
\begin{tikzpicture}[thick]

\begin{scope}[xshift = 0cm]
\draw (0,0) -- (-1,1);
\draw (0,0) -- (1,1);
\draw (0,0) -- (0,-1);
\draw (0,1) -- (-0.5, 0.5);

\draw (2,0) node{$\overset{\simeq}{\longrightarrow}$};

\begin{scope}[xshift = 4cm]
\draw (0,0) -- (-1,1);
\draw (0,0) -- (1,1);
\draw (0,0) -- (0,-1);
\draw (0,1) -- (0.5, 0.5);
\end{scope}
\end{scope}

\begin{scope}[xshift = 8cm]
\draw (0,0) -- (-1,1)node[above]{$a$}node[pos = 0.25, anchor = north east]{$f$};
\draw (0,0) -- (1,1)node[above]{$c$};
\draw (0,0) -- (0,-1)node[below]{$d$};
\draw (0,1)node[above]{$b$} -- (-0.5, 0.5);

\draw (2,0) node{$\overset{F^{abc}_{d;ef}}{\longrightarrow}$};

\begin{scope}[xshift = 4cm]
\draw (0,0) -- (-1,1)node[above]{$a$};
\draw (0,0) -- (1,1)node[above]{$c$}node[pos = 0.25, anchor = north west]{$e$};
\draw (0,0) -- (0,-1)node[below]{$d$};
\draw (0,1)node[above]{$b$} -- (0.5, 0.5);
\end{scope}
\end{scope}

\end{tikzpicture}
\caption{(Left) the associator natural isomorphism; (Right) the matrix elements of the associator natural isomorphism in the basis given by labels of internal edges.}\label{fig:T_graph_associator}
\end{figure}
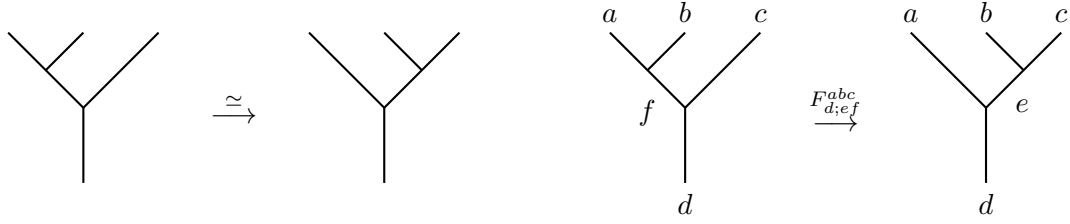

By considering a more restricted class of binary trees, we can relax the condition of $T$ to also allow $T$ to be a functor from $\mcC$ to another fusion category $\mcD$. Basically, a dot is still interpreted as $T$ and a \lq$Y$'-shape graph is interpreted either as the tensor product in $\mcC$ or as that in $\mcD$ depending on the specific configuration. Labels of edges will also be simple objects from either $\L(\mcC)$ or $\L(\mcD)$. The horizontal juxtaposition/vertical concatenation rules remain the same as before. We will not give an explicit characterization for such class of trees since all such trees we will be using in next subsection are simple enough so that the readers can easily verify that those  allow $T$ to take a target different from $\mcC$.

\subsection{Tensor Functors}\label{4.2}
Since tensor functor is a fundamental notion in our theory, we provide a skeletal definition in this subsection. We will use diagrammatic notations introduced in the previous subsection.

Let $\mcC $ and $\mcD$ be multiplicity-free fusion categories and without loss of generality, assume that the left/right unit isomorphisms $\unit \otimes X \simeq X \simeq X \otimes \unit$ in $\mcC$ and $\mcD$ are all the identity maps. We will use English letters to denote simple objects and Greek letters to denote integers.

Recall that a tensor functor $(T,T_2,T_0): \mcC \to \mcD$ is a linear strong monoidal functor that is faithful and exact. Such a functor determines, and is determined by, the data $(T_{ab}, T^{ab}_{c; (f\gamma),(d\alpha;e\beta)})$ satisfying Equations \ref{equ:T_matrix_size}, \ref{equ:T_unit}, \ref{equ:T_matrix_element}, and \ref{equ:T_unit_element}. 

For each $b \in \L(\mcC)$,
\begin{align}
\label{equ:T_obj}
    T(b)&= \oplus_{a \in \L(\mcD)} T_{ab} \, a, \quad T_{ab} \in \mathbb{N}.
\end{align}
For each $a,b \in \L(\mcC)$, there exists a natural isomorphism,
\begin{align}
    T_2(a,b): T(a) \otimes T(b) \quad \overset{\simeq}{\longrightarrow} \quad T(a \otimes b).
\end{align}
The isomorphism $T_2(a,b)$ determines a family of invertible matrices $\{T^{ab}_{c}\,\colon\, a,b \in \L(\mcC), c \in \L(\mcD) \}$, where the matrix elements of $T_{c}^{ab}$ are given by
\begin{align}
    \{T^{ab}_{c; (f\gamma),(d\alpha;e\beta)}\,\colon \, f \in \L(\mcC), d,e \in \L(\mcD), 1 \leq \alpha \leq T_{da},1 \leq \beta \leq T_{eb},1 \leq \gamma \leq T_{cf} \},
\end{align}
or diagrammatically shown in Figure \ref{fig:T_2_matrix_elements}. In particular, $T^{ab}_{c}$ being an isomorphism implies that the number of rows and the number of columns coincide,
\begin{align}
\label{equ:T_matrix_size}
    \sum_{d,e} T_{da}T_{eb}N_{de}^c &= \sum_{f} N_{ab}^f T_{cf}, \quad, \forall a,b \in \L(\mcC), c \in \L(\mcD).
\end{align}

\begin{figure}
\centering
\begin{tikzpicture}[thick]

\begin{scope}[xshift = 0cm]
\draw (0,0) -- (-1,1)node[above]{$a$}node[pos = .5, anchor = south west]{$\alpha$} node[pos = .25, anchor = north east]{$d$};
\draw (0,0) -- (1,1)node[above]{$b$}node[pos = .5, anchor = south east]{$\beta$}node[pos = .25, anchor = north west]{$e$};
\draw (0,0) -- (0,-1)node[below]{$c$};

\fill (-0.5,0.5) circle[radius = 2pt];
\fill (0.5,0.5) circle[radius = 2pt];
\end{scope}

\draw[->] (2,0) -- (4,0)node[pos = .5, above]{$T^{ab}_{c;(f\gamma),(d\alpha;e\beta)}$};

\begin{scope}[xshift = 6cm]
\draw (0,0) -- (-1,1)node[above]{$a$};
\draw (0,0) -- (1,1)node[above]{$b$};
\draw (0,0) -- (0,-1)node[below]{$c$}node[pos = .5, anchor = west]{$\gamma$}node[pos = .25, anchor = east]{$f$};

\fill (0,-0.5) circle[radius = 2pt];
\end{scope}

\end{tikzpicture}
\caption{The matrix elements of $T^{ab}_{c}$.}\label{fig:T_2_matrix_elements}
\end{figure}
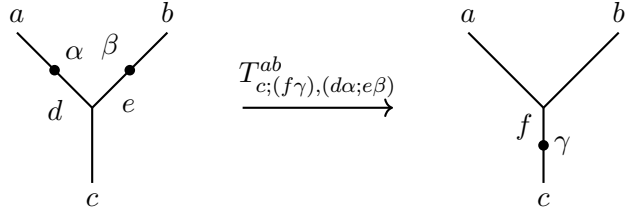

Without loss of generality, We assume $T_0: \unit \rightarrow T(\unit)$ is the identity map. Thus,
\begin{align}
\label{equ:T_unit}
   T_{a \unit} = \delta_{a, \unit}. 
\end{align}

The axioms $T_2$ needs to satisfy are shown in Figures \ref{fig:T_2_hexagon}, \ref{fig:T_2_triangle1}, and \ref{fig:T_2_triangle2}, where we omitted labels of internal edges. Correspondingly, we obtain the following equations.

\begin{align}
\label{equ:T_matrix_element}
   \sum_{l,\lambda} T^{ab}_{h; (i\omega),(e \alpha;f \beta)} T^{ic}_{d; (j \theta),(h \omega; g \gamma)} F^{abc}_{j;ki}= 
   \sum_{i, \omega}F^{efg}_{d;lh} T^{bc}_{l;(k \lambda),(f \beta; g \gamma)} T^{ak}_{d;(j \theta),(e \alpha;l \lambda)},\nonumber \\
 \forall a,b,c \in \L(\mcC), \ d,e,f,g,h,j,k \in \L(\mcD),\\
  1\leq \alpha \leq T_{ea},  1 \leq \beta \leq T_{fb}, 1 \leq  \gamma \leq T_{gc}, 1 \leq \theta \leq T_{dj}. \nonumber
\end{align}

\begin{align}
\label{equ:T_unit_element}
    T^{\unit a}_{b; (a \beta),(\unit 1;b \alpha)} = T^{a \unit}_{b; (a \beta),(b \alpha;\unit 1)} = \delta_{\alpha, \beta}. \nonumber\\
    \forall a \in \L(\mcC), b \in \L(\mcD), 1 \leq \alpha \leq T_{ba}.
\end{align}
Note that in Equation \ref{equ:T_matrix_element}, the $F$-matrix on the left hand side is computed in $\mcC$ while the $F$-matrix on the right hand side is computed in $\mcD$.

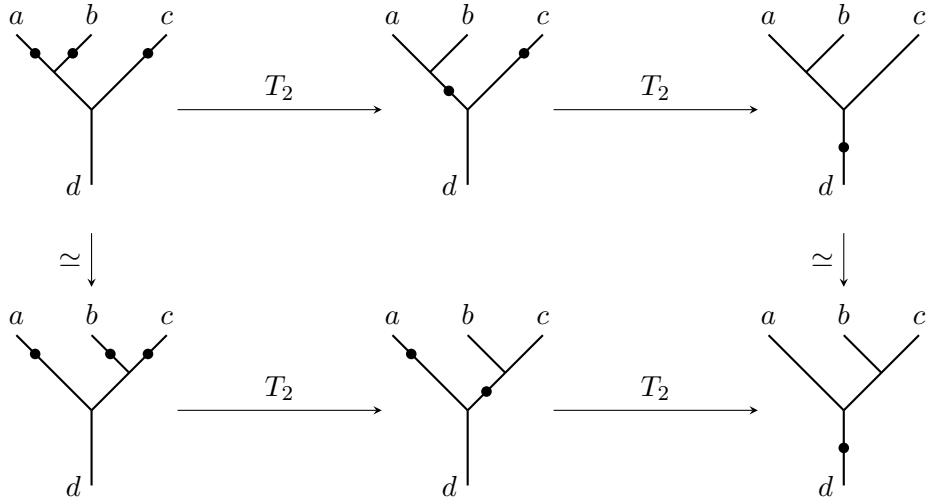
\begin{figure}
\centering
\begin{tikzpicture}[thick]

\begin{scope}[xshift = 0cm, yshift = 0cm]
\draw (0,0) -- (-1,1)node[above]{$a$};
\draw (0,0) -- (1,1)node[above]{$c$};
\draw (0,0) -- (0,-1)node[left]{$d$};
\draw (0,1)node[above]{$b$} -- (-0.5, 0.5);

\fill (-0.75, 0.75) circle[radius = 2pt];
\fill (-0.25, 0.75) circle[radius = 2pt];
\fill (0.75, 0.75) circle[radius = 2pt];

\node (R_One) at (1,0) {};
\node (B_One) at (0,-1.5) {};
\end{scope}

\begin{scope}[xshift = 5cm, yshift = 0cm]
\draw (0,0) -- (-1,1)node[above]{$a$};
\draw (0,0) -- (1,1)node[above]{$c$};
\draw (0,0) -- (0,-1)node[left]{$d$};
\draw (0,1)node[above]{$b$} -- (-0.5, 0.5);

\fill (-0.25, 0.25) circle[radius = 2pt];
\fill (0.75, 0.75) circle[radius = 2pt];

\node (R_Two) at (1,0) {};
\node (L_Two) at (-1,0) {};
\end{scope}

\begin{scope}[xshift = 10cm, yshift = 0cm]
\draw (0,0) -- (-1,1)node[above]{$a$};
\draw (0,0) -- (1,1)node[above]{$c$};
\draw (0,0) -- (0,-1)node[left]{$d$};
\draw (0,1)node[above]{$b$} -- (-0.5, 0.5);

\fill (0, -0.5) circle[radius = 2pt];

\node (B_Three) at (0,-1.5) {};
\node (L_Three) at (-1,0) {};
\end{scope}

\begin{scope}[xshift = 0cm, yshift = -4cm]
\draw (0,0) -- (-1,1)node[above]{$a$};
\draw (0,0) -- (1,1)node[above]{$c$};
\draw (0,0) -- (0,-1)node[left]{$d$};
\draw (0,1)node[above]{$b$} -- (0.5, 0.5);

\fill (-0.75, 0.75) circle[radius = 2pt];
\fill (0.25, 0.75) circle[radius = 2pt];
\fill (0.75, 0.75) circle[radius = 2pt];

\node (R_Four) at (1,0) {};
\node (T_Four) at (0,1.5) {};
\end{scope}

\begin{scope}[xshift = 5cm, yshift = -4cm]
\draw (0,0) -- (-1,1)node[above]{$a$};
\draw (0,0) -- (1,1)node[above]{$c$};
\draw (0,0) -- (0,-1)node[left]{$d$};
\draw (0,1)node[above]{$b$} -- (0.5, 0.5);

\fill (-0.75, 0.75) circle[radius = 2pt];
\fill (0.25, 0.25) circle[radius = 2pt];

\node (R_Five) at (1,0) {};
\node (L_Five) at (-1,0) {};
\end{scope}

\begin{scope}[xshift = 10cm, yshift = -4cm]
\draw (0,0) -- (-1,1)node[above]{$a$};
\draw (0,0) -- (1,1)node[above]{$c$};
\draw (0,0) -- (0,-1)node[left]{$d$};
\draw (0,1)node[above]{$b$} -- (0.5, 0.5);

\fill (0, -0.5) circle[radius = 2pt];

\node (T_Six) at (0,1.5) {};
\node (L_Six) at (-1,0) {};
\end{scope}

\foreach \from/\to in {R_One/L_Two, R_Two/L_Three,   R_Four/L_Five, R_Five/L_Six} \draw[->, > = stealth, thin] (\from) -- (\to)node[pos = .5, above]{$T_2$}; 
\foreach \from/\to in {B_One/T_Four,B_Three/T_Six} \draw[->, > = stealth, thin] (\from) -- (\to)node[pos = .5, left]{$\simeq$};

\end{tikzpicture}
\caption{The hexagon equation for $T_2$.}\label{fig:T_2_hexagon}
\end{figure}

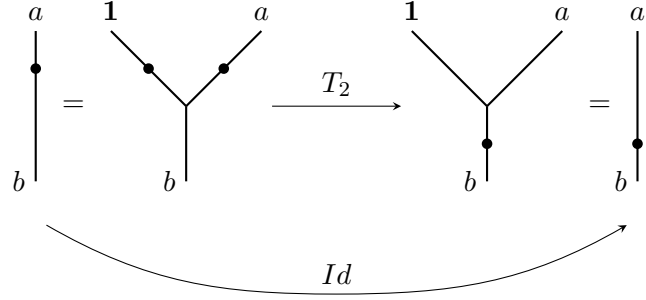
\begin{figure}
\centering
\begin{tikzpicture}[thick]
\begin{scope}[xshift = 0cm, yshift = 0cm]
\draw (0,-1)node[left]{$b$} -- (0,1)node[above]{$a$};
\fill (0,0.5) circle[radius = 2pt];
\draw (0.5, 0) node{$=$};
\begin{scope}[xshift = 2cm]
\draw (0,0) -- (-1,1)node[above]{$\unit$};
\draw (0,0) -- (1,1)node[above]{$a$};
\draw (0,0) -- (0,-1)node[left]{$b$};

\fill (-0.5,0.5) circle[radius = 2pt];
\fill (0.5,0.5) circle[radius = 2pt];

\node (R_One) at (1,0) {};
\end{scope}
\node (B_One) at (0,-1.5) {};
\end{scope}

\begin{scope}[xshift = 6cm, yshift = 0cm]
\draw (0,0) -- (-1,1)node[above]{$\unit$};
\draw (0,0) -- (1,1)node[above]{$a$};
\draw (0,0) -- (0,-1)node[left]{$b$};

\fill (0,-0.5) circle[radius = 2pt];

\draw (1.5, 0) node{$=$};

\draw (2,-1)node[left]{$b$} -- (2,1)node[above]{$a$};
\fill (2,-0.5) circle[radius = 2pt];

\node (L_Two) at (-1,0){};
\node (B_Two) at (2,-1.5){};

\end{scope}

\draw[->, >=stealth,  thin] (R_One) -- (L_Two)node[pos = .5, above]{$T_2$};
\draw[->, >=stealth, thin] (B_One)[out = -30, in = 180] to (4, -2.5)node[above]{$Id$}[out = 0, in = -150] to (B_Two);
\end{tikzpicture}
\caption{The first triangle equation for $T_2$}\label{fig:T_2_triangle1}
\end{figure}

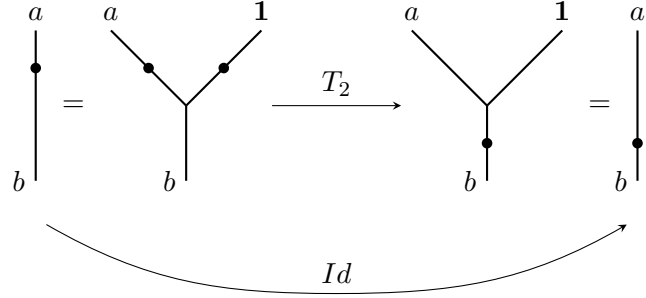
\begin{figure}
\centering
\begin{tikzpicture}[thick]
\begin{scope}[xshift = 0cm, yshift = 0cm]
\draw (0,-1)node[left]{$b$} -- (0,1)node[above]{$a$};
\fill (0,0.5) circle[radius = 2pt];
\draw (0.5, 0) node{$=$};
\begin{scope}[xshift = 2cm]
\draw (0,0) -- (-1,1)node[above]{$a$};
\draw (0,0) -- (1,1)node[above]{$\unit$};
\draw (0,0) -- (0,-1)node[left]{$b$};

\fill (-0.5,0.5) circle[radius = 2pt];
\fill (0.5,0.5) circle[radius = 2pt];

\node (R_One) at (1,0) {};
\end{scope}
\node (B_One) at (0,-1.5) {};
\end{scope}

\begin{scope}[xshift = 6cm, yshift = 0cm]
\draw (0,0) -- (-1,1)node[above]{$a$};
\draw (0,0) -- (1,1)node[above]{$\unit$};
\draw (0,0) -- (0,-1)node[left]{$b$};

\fill (0,-0.5) circle[radius = 2pt];

\draw (1.5, 0) node{$=$};

\draw (2,-1)node[left]{$b$} -- (2,1)node[above]{$a$};
\fill (2,-0.5) circle[radius = 2pt];

\node (L_Two) at (-1,0){};
\node (B_Two) at (2,-1.5){};

\end{scope}

\draw[->, >=stealth,  thin] (R_One) -- (L_Two)node[pos = .5, above]{$T_2$};
\draw[->, >=stealth, thin] (B_One)[out = -30, in = 180] to (4, -2.5)node[above]{$Id$}[out = 0, in = -150] to (B_Two);
\end{tikzpicture}
\caption{The second triangle equation for $T_2$}\label{fig:T_2_triangle2}
\end{figure}

\subsection{Hopf Monads}\label{4.3}
We give a skeletal definition of a left Hopf monad. The corresponding definition for a right Hopf monad can be carried out analogously. Again let $\mcC$ be a rank=$n$ multiplicity-free fusion category with a complete set of representatives $\L = \L(\mcC) = \{a,b,c, \cdots\}$.

Let $(T, H, T_0, \eta)$ be a left Hopf monad on $\mcC$ according to \hyperref[dfn1]{\textbf{Definition 1}}. $T$ is a functor and is determined by the $n \times n$ integer-valued matrix $(T_{ab})$:
\begin{align}
    T(b)  = \oplus_{a \in \L} T_{ab}\, a, \quad b\in\L. 
\end{align}

$H$ is a natural isomorphism from $T \circ \otimes \circ (Id \times T)$ to $\otimes \circ (T \times T)$. The diagrammatical representations for the above two functors are represented in Figure \ref{fig:T_graph_examples} (IV) and (V), respectively. Hence $H$ is a collection of invertible matrices $\{H_{c}^{ab}\,\colon\, a,b,c \in \L\}$, where the matrix elements of $H^{ab}_{c}$ are given by
\begin{align}
    \{H^{ab}_{c;(f\gamma;g\theta),(d\alpha;e\beta)}\,\colon\, d,e,f,g \in \L, 1 \leq \alpha \leq T_{db}, 1 \leq \beta \leq T_{ce}, 1 \leq \gamma \leq T_{fa}, 1 \leq \theta \leq T_{gb} \}, 
\end{align}
or diagrammatically illustrated in Figure \ref{fig:H_matrix_elements}. In particular, the number of rows and the number of columns of $H^{ab}_{c}$ must be equal:
\begin{align}
\label{equ:H_matrix_size}
    \sum_{f,g \in L} T_{fa}T_{gb}N_{fg}^c &= \sum_{d,e \in L} T_{db}T_{ce}N_{ad}^e.
\end{align}

The matrix elements for $T_0\colon T(\unit) \to \unit$ is determined by a vector
\begin{align}
\label{equ:T_0_elements}
\{\epsilon_{\alpha}\,\colon\, 1 \leq \alpha \leq T_{\unit\unit}\},    
\end{align}
and the natural transformation $\eta\colon 1_\mcC \to T $ gives a family of vectors $\{\eta_{a}\,\colon\, a \in \L\}$ where $\eta_{a}$ is 
\begin{align}
    \label{equ:eta_elements}
    \{\eta_{a,\alpha}\,\colon\, 1 \leq \alpha \leq T_{aa}\}.
\end{align}
See Figure \ref{fig:T_0_eta_matrix_elements} for a diagrammatic definition of the vectors $\epsilon$ and $\eta_{a}$.

\begin{figure}
\centering
\begin{tikzpicture}[thick]
\begin{scope}[xshift = 0cm]
\draw (0,2-0)node[above]{$a$} -- (1,2-1);
\draw (2,2-0)node[above]{$b$} -- (1,2-1)node[pos = .5, anchor = south east]{$\alpha$}node[pos = .75, anchor = north west ]{$d$};
\draw (1,2-2)node[below]{$c$} -- (1,2-1)node[pos = .5, right]{$\beta$}node[pos = .75, left]{$e$};
\fill (1.5,2-0.5) circle[radius = 2pt];
\fill (1,2-1.5) circle[radius = 2pt];
\end{scope}

\begin{scope}[xshift = 2.5cm]
\draw[->, >=stealth] (0,2-1) -- (3,2-1)node[pos = .5, above]{$H^{ab}_{c;(f\gamma;g\theta),(d\alpha;e\beta)}$};
\end{scope}

\begin{scope}[xshift = 6cm]
\draw (0,2-0)node[above]{$a$} -- (1,2-1)node[pos = .5, anchor = south west]{$\gamma$}node[pos = .75, anchor = north east]{$f$};
\draw (2,2-0)node[above]{$b$} -- (1,2-1)node[pos = .5, anchor = south east]{$\theta$}node[pos = .75, anchor = north west ]{$g$};
\draw (1,2-2)node[below]{$c$} -- (1,2-1);
\fill (1.5,2-0.5) circle[radius = 2pt];
\fill (0.5,2-0.5) circle[radius = 2pt];
\end{scope}
\end{tikzpicture}
\caption{The matrix elements of $H^{ab}_{c}$}\label{fig:H_matrix_elements}
\end{figure}
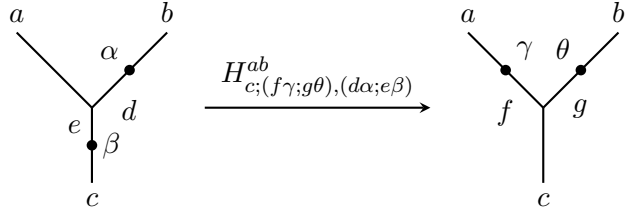

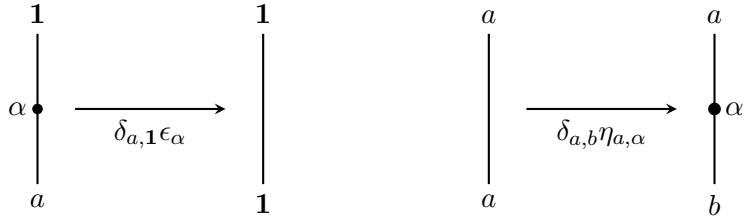
\begin{figure}
\centering
\begin{tikzpicture}[thick]
\begin{scope}[xshift = 0cm]
\draw (0,2-0)node[above]{$\unit$} -- (0,2-2) node[below]{$a$} node[pos = .5, left]{$\alpha$};
\fill (0,2-1) circle[radius = 2pt];

\draw[->, >=stealth] (0.5, 2-1) --(2.5,2-1) node[pos = .5, below]{$\delta_{a,\unit}\epsilon_{\alpha}$};

\draw (3,2-0)node[above]{$\unit$}--(3,2-2)node[below]{$\unit$};
\end{scope}

\begin{scope}[xshift = 6cm]
\draw (0,2-0)node[above]{$a$} -- (0,2-2) node[below]{$a$};
\draw[->, >=stealth] (0.5, 2-1) --(2.5,2-1) node[pos = .5, below]{$\delta_{a,b}\eta_{a,\alpha}$};

\draw (3,2-0)node[above]{$a$}--(3,2-2)node[below]{$b$} node[pos = .5, right]{$\alpha$};
\draw[fill = black] (3,2-1) circle[radius = 2pt];
\end{scope}
\end{tikzpicture}
\caption{The matrix elements of $T_0$ (Left) and $\eta$ (Right).}\label{fig:T_0_eta_matrix_elements}
\end{figure}

The heptagon axiom for $H$ is the equality of two paths of natural transformations from the functor $T\circ \otimes \circ (Id \times T\circ \otimes)\circ(Id \times Id \times T)$ to $\otimes \circ (Id \times \otimes)\circ(T \times T \times T)$. The diagrammatic  representation of the axiom is presented in Figure \ref{fig:heptagon}, from which we derive the Heptagon Equation:
\begin{equation}
\begin{split}
\label{equ:heptagon_matrix_elements}
&{}\quad\ \,\sum_{\theta} H^{af}_{d;(g'{\beta}';h'\theta),(g{\beta};h{\gamma})}\,H^{bc}_{h';(e'{\alpha}';f'{\gamma}'),(e{\alpha};fw)} \quad = \\
&\sum_{\substack{\tilde{e}, \tilde{f}, \tilde{g}, \hat{f}  \tilde{\alpha}, \tilde{\beta}, \hat{\beta}}} H^{bc}_{g;(\tilde{e}\tilde{{\alpha}};\tilde{f}\tilde{{\beta}}),(e{\alpha};f{\beta})} \, (F^{a\tilde{e}\tilde{f}}_{h})^{-1}_{\tilde{g}g}\, H^{\tilde{g}c}_{d;(\hat{f}\hat{{\beta}};f'{\gamma}'),(\tilde{f}\tilde{{\beta}};h{\gamma})}
 H^{ab}_{\hat{f};(g'{\beta}';e'{\alpha}'),(\tilde{e}\tilde{{\alpha}};\tilde{g}\hat{{\beta}})}\, F^{g'e'f'}_{d;h'\hat{f}} 
\end{split}
\end{equation}

Conditions in Equation \ref{equ:HM_triangle_axioms} can be represented in a similar way by diagrams. As an example, we illustrate the diagrams for the first two equalities in Figure \ref{fig:HM_trangle1} and \ref{fig:HM_trangle2}, respectively, and leave the rest as an exercise. The four equalities in Equation \ref{equ:HM_triangle_axioms} are written, in terms of matrix elements, in the following equations. 
\begin{align}
    \label{equ:HM_triangle_matrix_elements}
    \begin{split}
        \sum_{{  { \beta}}} \eta_{c,{  { \beta}}}\ H^{ab}_{c; (e{ \gamma};f{ \theta}),(d{ \alpha};c{  { \beta}})} = \delta_{f,d}\delta_{e,a}\delta_{{ \theta},{ \alpha}}\eta_{a,{ \gamma}}, \quad & \quad  \sum_{{  { \alpha}}} \eta_{\unit, {  { \alpha}}} \ \epsilon_{{  { \alpha}}} =1, \\
         \sum_{{  { \theta}}} H^{a\unit}_{b; (b{ \gamma};\unit {  { \theta}}),(c{ \alpha};d{ \beta})} \ \epsilon_{{  { \theta}}} = \delta_{c,\unit}\delta_{d,a}\delta_{{ \gamma},{ \beta}}\epsilon_{{ \alpha}}, \quad & \quad \sum_{{  { \beta}},{  { \gamma}}} H^{\unit a}_{b; (\unit {  { \gamma}}; b { \theta}),(a {  { \beta}}; a { \alpha})} \  \eta_{a {  { \beta}}} \ \epsilon_{{  { \gamma}}} = \delta_{{ \theta},{ \alpha}}.
    \end{split}
\end{align}

In summary, a left Hopf monad is defined by the data $\{T_{ab}, H^{ab}_{c;(f\gamma;g\theta),(d\alpha;e\beta)}, \epsilon_{\alpha}, \eta_{a, \alpha} \}$ satisfying Equations \ref{equ:H_matrix_size}, \ref{equ:heptagon_matrix_elements}, and \ref{equ:HM_triangle_matrix_elements}.

\begin{figure}
\centering
\begin{tikzpicture}[thick]
\begin{scope}[xshift = -1.5cm, yshift = 4.75cm]
\draw (0,2.5-0)node[above]{$a$} -- (1.5,2.5-1.5)node[pos = 0.375, anchor = north east](out17){};
\draw (3,2.5-0)node[above]{$c$} -- (1.5,2.5-1.5)node[pos = .25, anchor = south east]{${\alpha}$}node[pos = .75, anchor = south east ]{${\beta}$}node[pos = 0.375, anchor = north west](out12){$e$}node[pos = 0.625, anchor = north west]{$f$}node[pos = 0.875, anchor = north west]{$g$};
\draw (1.5,2.5-2.5)node[below]{$d$} -- (1.5,2.5-1.5)node[pos = .5, right]{${\gamma}$}node[pos = .75, left]{$h$};
\draw (1.5,2.5-0)node[above]{$b$} -- (2.25, 2.5-0.75);
\fill (1.5+0.375,2.5-1.5+0.375) circle[radius = 2pt];
\fill (3-0.375,2.5-0.375) circle[radius = 2pt];
\fill (1.5, 2.5-2)circle[radius = 2pt];
\end{scope}

\begin{scope}[xshift = 3.19099cm, yshift = 2.49094cm]
\draw (0,2.5-0)node[above]{$a$} -- (1.5,2.5-1.5)node[pos = .375, anchor = north east]{};
\draw (3,2.5-0)node[above]{$c$} -- (1.5,2.5-1.5)node[pos = .25, anchor = south east]{$\tilde{{\beta}}$}node[pos = 0.375, anchor = north west](out23){$\tilde{f}$}node[pos = 0.75, anchor = north west]{$g$};
\draw (1.5,2.5-0)node[above](in12){$b$} -- (2.25, 2.5-0.75)node[pos = .5, anchor = east]{$\tilde{{\alpha}}$}node[pos = .75, anchor = north east]{$\tilde{e}$};
\draw (1.5,2.5-2.5)node[below]{$d$} -- (1.5,2.5-1.5)node[pos = .5, right]{${\gamma}$}node[pos = .75, left]{$h$};
\fill (1.5+0.375,2.5-0.375) circle[radius = 2pt];
\fill (3-0.375,2.5-0.375) circle[radius = 2pt];
\fill (1.5, 2.5-2)circle[radius = 2pt];
\end{scope}

\begin{scope}[xshift = 4.34957cm, yshift = -2.58513cm]
\draw (0,2.5-0)node[above]{$a$} -- (1.5,2.5-1.5)node[pos = .75, anchor = north east](out34){$\tilde{g}$};
\draw (3,2.5-0)node[above]{$c$} -- (1.5,2.5-1.5)node[pos = .25, anchor = south east]{$\tilde{{\beta}}$}node[pos = 0.375, anchor = north west]{$\tilde{f}$};
\draw (1.5,2.5-0)node[above](in23){$b$} -- (0.75, 2.5-0.75)node[pos = .5, anchor = south east]{$\tilde{{\alpha}}$}node[pos = .75, anchor = north west]{$\tilde{e}$};
\draw (1.5,2.5-2.5)node[below]{$d$} -- (1.5,2.5-1.5)node[pos = .5, right]{${\gamma}$}node[pos = .75, left]{$h$};
\fill (1.5-0.375,2.5-0.375) circle[radius = 2pt];
\fill (3-0.375,2.5-0.375) circle[radius = 2pt];
\fill (1.5, 2.5-2)circle[radius = 2pt];
\draw (0.75,2.5-0)node{};
\end{scope}

\begin{scope}[xshift = 1.1033cm, yshift = -6.65581cm]
\draw (0,2.5-0)node[above]{$a$} -- (1.5,2.5-1.5)node[pos = .625, anchor = north east](out45){$\tilde{g}$}node[pos = .75, anchor = south west]{$\hat{{\beta}}$}node[pos = .875, anchor = north east]{$\hat{f}$};
\draw (3,2.5-0)node[above]{$c$} -- (1.5,2.5-1.5)node[pos = .25, anchor = south east]{${\gamma}'$}node[pos = 0.375, anchor = north west]{$f'$}node[pos = .75, anchor = north west]{};
\draw (1.5,2.5-0)node[above](in34){$b$} -- (0.75, 2.5-0.75)node[pos = .5, anchor = west]{$\tilde{{\alpha}}$}node[pos = .75, anchor = south east]{$\tilde{e}$};
\draw (1.5,2.5-2.5)node[below]{$d$} -- (1.5,2.5-1.5);
\fill (1.5-0.375,2.5-0.375) circle[radius = 2pt];
\fill (3-0.375,2.5-0.375) circle[radius = 2pt];
\fill (1.5-0.375, 2.5-1.5+0.375)circle[radius = 2pt];
\end{scope}

\begin{scope}[xshift = -4.1033cm, yshift = -6.65581cm]
\draw (0,2.5-0)node[above]{$a$} -- (1.5,2.5-1.5)node[pos = .25, anchor = south west]{${\beta}'$}node[pos = .375, anchor = north east]{$g'$}node[pos = .75, anchor = north east]{$\hat{f}$};
\draw (3,2.5-0)node[above]{$c$} -- (1.5,2.5-1.5)node[pos = .25, anchor = south east]{${\gamma}'$}node[pos = 0.625, anchor = north west](in45){$f'$};
\draw (1.5,2.5-0)node[above](out56){$b$} -- (0.75, 2.5-0.75)node[pos = .5, anchor = west]{${\alpha}'$}node[pos = .75, anchor = north west]{$e'$};
\draw (1.5,2.5-2.5)node[below]{$d$} -- (1.5,2.5-1.5);
\fill (1.5-0.375,2.5-0.375) circle[radius = 2pt];
\fill (3-0.375,2.5-0.375) circle[radius = 2pt];
\fill (0.375, 2.5-0.375)circle[radius = 2pt];
\end{scope}

\begin{scope}[xshift = -7.34957cm, yshift = -2.58513cm]
\draw (0,2.5-0)node[above]{$a$} -- (1.5,2.5-1.5)node[pos = .25, anchor = south west]{${\beta}'$}node[pos = .375, anchor = north east]{$g'$};
\draw (3,2.5-0)node[above]{$c$} -- (1.5,2.5-1.5)node[pos = .25, anchor = south east]{${\gamma}'$}node[pos = 0.375, anchor = north west]{$f'$}node[pos = .75, anchor = north west](in56){$h'$};
\draw (1.5,2.5-0)node[above](in76){$b$} -- (1.5+0.75, 2.5-0.75)node[pos = .5, anchor = east]{${\alpha}'$}node[pos = .75, anchor = north east]{$e'$};
\draw (1.5,2.5-2.5)node[below]{$d$} -- (1.5,2.5-1.5);
\fill (1.5+0.375,2.5-0.375) circle[radius = 2pt];
\fill (3-0.375,2.5-0.375) circle[radius = 2pt];
\fill (0.375, 2.5-0.375)circle[radius = 2pt];
\draw (2.25,0)node{};
\end{scope}

\begin{scope}[xshift = -6.19099cm, yshift = 2.49094cm]
\draw (0,2.5-0)node[above]{$a$} -- (1.5,2.5-1.5)node[pos = .25, anchor = south west]{${\beta}'$}node[pos = .375, anchor = north east](out76){$g'$};
\draw (3,2.5-0)node[above]{$c$} -- (1.5,2.5-1.5)node[pos = .25, anchor = south east]{${\alpha}$}node[pos = 0.375, anchor = north west]{$e$}node[pos = .625, anchor = north west]{$f$}node[pos = .75, anchor = south east]{$\theta$}node[pos = .875, anchor = north west]{$h'$};
\draw (1.5,2.5-0)node[above](in17){$b$} -- (1.5+0.75, 2.5-0.75);
\draw (1.5,2.5-2.5)node[below]{$d$} -- (1.5,2.5-1.5);
\fill (1.5+0.375,2.5-1.5+0.375) circle[radius = 2pt];
\fill (3-0.375,2.5-0.375) circle[radius = 2pt];
\fill (0.375, 2.5-0.375)circle[radius = 2pt];
\end{scope}

\draw[->, >=stealth,thin](out12)--(in12)node[pos = .5,anchor = south ]{$H$};
\draw[->, >=stealth,thin](out23)--(in23)node[pos = .5, anchor = west]{$F^{-1}$};
\draw[->, >=stealth,thin](out34)--(in34)node[pos = .5, anchor = north west]{$H$};
\draw[->, >=stealth,thin](out45)--(in45)node[pos = .5, below]{$H$};
\draw[->, >=stealth,thin](out56)--(in56)node[pos = .5, anchor = north east]{$F$};
\draw[->, >=stealth,thin](out17)--(in17)node[pos = .7, anchor = south east]{$H$};
\draw[->, >=stealth,thin](out76)--(in76)node[pos = .5, anchor = east]{$H$};

\end{tikzpicture}
\caption{Heptagon equation for $H$} \label{fig:heptagon}
\end{figure}
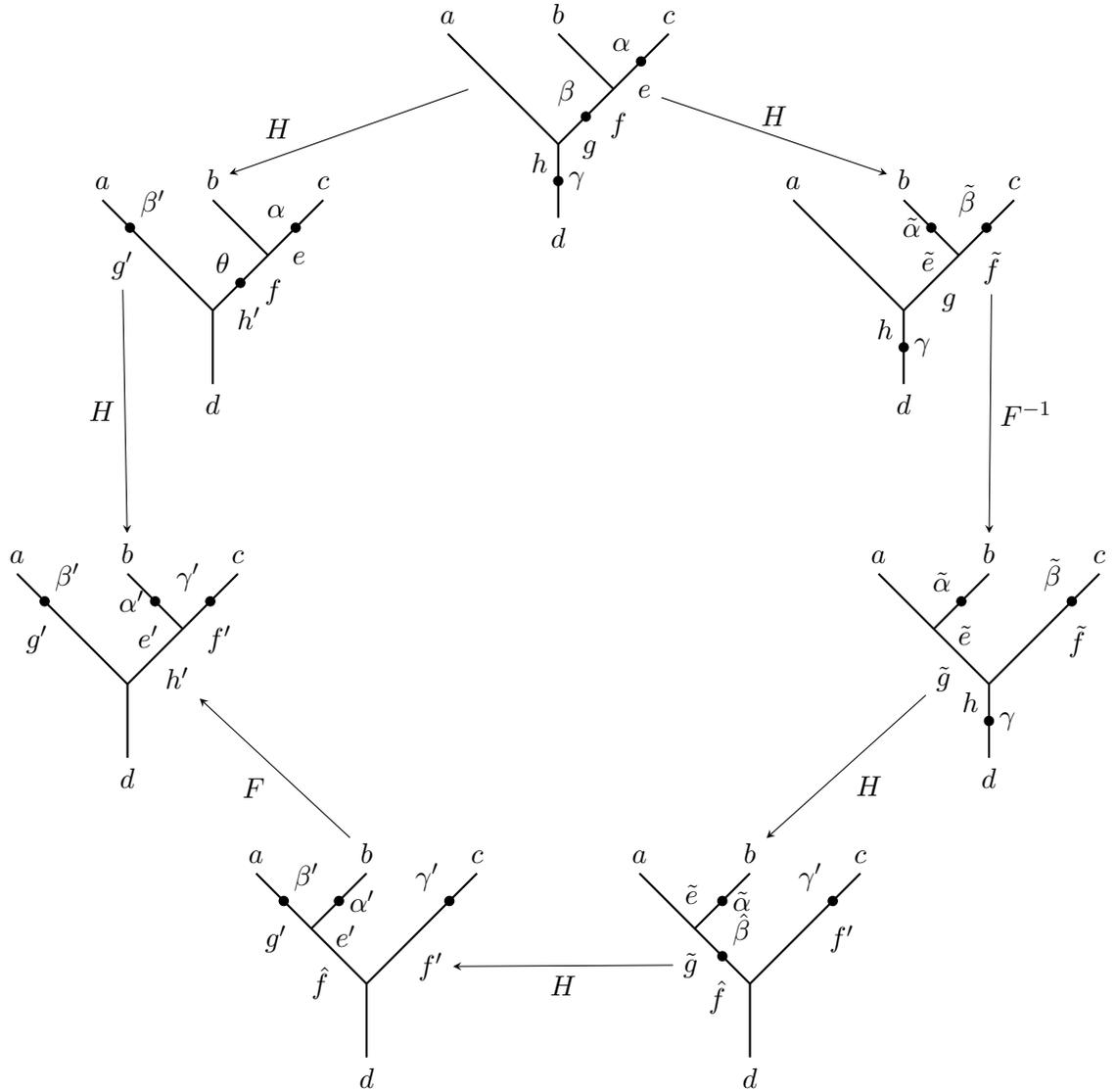

\begin{figure}
\centering
\begin{tikzpicture}[thick]
\begin{scope}[xshift = 0cm]
\draw (0,2-0)node[above]{$a$} -- (1,2-1);
\draw (2,2-0)node[above]{$b$} -- (1,2-1)node[pos = .5, anchor = south east]{$\alpha$}node[pos = .75, anchor = north west ]{$d$};
\draw (1,2-2)node[below]{$c$} -- (1,2-1);
\fill (1.5,2-0.5) circle[radius = 2pt];
\end{scope}

\begin{scope}[xshift = 5cm]
\draw (0,2-0)node[above]{$a$} -- (1,2-1);
\draw (2,2-0)node[above]{$b$} -- (1,2-1)node[pos = .5, anchor = south east]{$\alpha$}node[pos = .75, anchor = north west ]{$d$};
\draw (1,2-2)node[below]{$c$} -- (1,2-1)node[pos = .5, right]{$\beta$}node[pos = .75, left]{$c$};
\fill (1.5,2-0.5) circle[radius = 2pt];
\fill (1,2-1.5) circle[radius = 2pt];
\end{scope}

\begin{scope}[xshift = 10cm, yshift = 0cm]
\draw (0,2-0)node[above]{$a$} -- (1,2-1)node[pos = .5, anchor = south west]{${\gamma}$}node[pos = .75, anchor = north east]{$e$};
\draw (2,2-0)node[above]{$b$} -- (1,2-1)node[pos = .5, anchor = south east]{$\theta$}node[pos = .75, anchor = north west ]{$f$};
\draw (1,2-2)node[below]{$c$} -- (1,2-1);
\fill (1.5,2-0.5) circle[radius = 2pt];
\fill (0.5,2-0.5) circle[radius = 2pt];
\end{scope}

\draw[->, >=stealth,thin] (3-0.25,1)--(4+0.25,1)node[pos = .5, above]{$\eta$};
\draw[->, >=stealth,thin] (8-0.25,1)--(9+0.25,1)node[pos = .5, above]{$H$};
\draw[->, >=stealth,thin] (3-0.25,0)[out = -30, in = 180]to (6,-1)[out = 0, in = -150]node[below]{$\eta$} to(9+0.25,0);

\end{tikzpicture}
\caption{The first equality in Equation \ref{equ:HM_triangle_axioms}}\label{fig:HM_trangle1}
\end{figure}
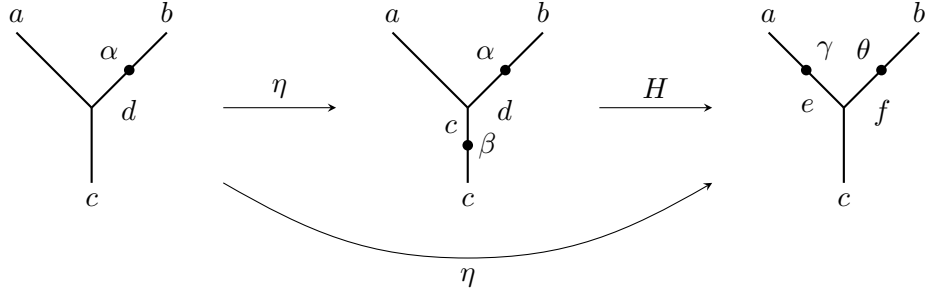

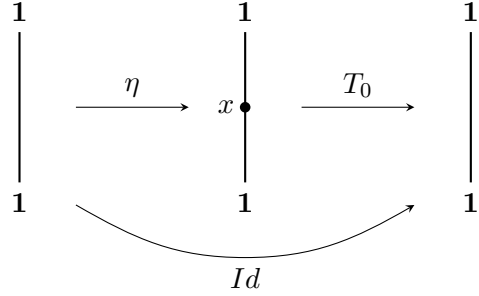
\begin{figure}
\centering
\begin{tikzpicture}[thick]
\begin{scope}
\draw (0,2-0)node[above]{$\unit$} -- (0,2-2) node[below]{$\unit$}node[pos = .5, left]{};
\end{scope}
\begin{scope}[xshift = 3cm]
\draw (0,2-0)node[above]{$\unit$} -- (0,2-2) node[below]{$\unit$}node[pos = .5, left]{$x$};
\fill (0,2-1) circle[radius = 2pt];
\end{scope}
\begin{scope}[xshift = 6cm, yshift = 0cm]
\draw (0,2-0)node[above]{$\unit$} -- (0,2-2) node[below]{$\unit$}node[pos = .5, left]{};
\end{scope}

\draw[->, >=stealth,thin] (1-0.25, 1) -- (2+0.25,1)node[pos = .5, above]{$\eta$};
\draw[->, >=stealth,thin] (4-0.25, 1) -- (5+0.25,1)node[pos = .5, above]{$T_0$};
\draw[->, >=stealth,thin] (1-0.25,-0.3)[out = -30, in = 180] to (3,-1)[out = 0, in = -150]node[below]{$Id$} to (5+0.25,-0.3);
\end{tikzpicture}
\caption{The second equality in Equation \ref{equ:HM_triangle_axioms}}\label{fig:HM_trangle2}
\end{figure}

\section{Classification of Hopf monads}\label{5}
In this section, we consider HMs in $\VecZp$ as well as those from condensations.  For $p=1$, the classification of HMs is reduced to the classification of Hopf algebras over $\mbbC$, and this most trivial case shows the difficulty of classifying all HMs. Hence, instead of classifying all HMs, we attempt to characterize them and find many interesting examples.  At the same time, we will focus on some examples and describe their module categories, where the classification of indecomposable module categories over $\Vec_G$ (see e,g. \cite[2.7]{etingof2009fusion}) shall be used. Every indecomposable module category over $\Vec_G$ is essentially determined by 
\begin{itemize}
    \item a group algebra $\mbbC[K]$ for a subgroup $K \subset G$,
    \item a $2$-cocycle $\psi$ on $K$ determining the multiplication in $\mbbC[K]$. 
\end{itemize}
The module category associated with $(K,\psi)$ has equivalence classes of simple objects corresponding to the left cosets $\faktor{G}{K}$ and the $\Vec_G$ action is the obvious one: $g. [aK] = [gaK]$.

\subsection{Characterization of HMs on $\VecZp$ for $p$ prime}\label{5.1}

\subsubsection{$\Vec$}\label{5.1.1}
A HM $T$ is determined by its action on the unit object and $T(\unit)=\mcH \implies T(X)=\mcH \otimes X$. It is then easy to observe that the Hopf monad axioms, with all objects present replaced by the unit, imply that $\mcH$ is a Hopf algebra. Therefore, all possible HMs on $\Vec$ are determined by an \textit{extension} of the Hopf algebra structure. If one assumes the HM structural morphisms ($\mu,\eta,\ldots$) to be additive as well, then the $T$ structure is entirely determined by $\mcH$ as all objects in $\Vec$ are sums of copies of unit.   That is, any $T$ is of the form $\mcH \otimes -$.

\begin{remark}\label{rmk2}
As explained above, in case a HM is of the form $T(-)=\mcH \otimes -$ at the object level, it does not necessarily imply that the HM structure is entirely determined by the structure of $\mcH$. From now on, all HMs of the form $T(-)=\mcH \otimes -$ for a Hopf algebra $\mcH$ are assumed to have their structure inherited from that of $\mcH$.
\end{remark}

\subsubsection{$\Vec_{\mbbZ_2}=\{\unit,e\}$} \label{5.1.2}
In general for $\Vec_G$, it can be observed that a Hopf algebra is simply a $G$-graded Hopf algebra $\mcH=\oplus \mcH_g$. So for $G=\mbbZ_2$, a Hopf algebra is an extension of the structure on a super-algebra. We can not characterize any further Hopf algebras, as its most trivial case contains all Hopf algebras over $\Vec$ (when $\mcH_g=0$ for $g\neq 1$). But we will examine the case of $2+e$ which is applied in an example of condensation.  Note that here $e$ denotes the non-trivial simple object in $\Vec_{\mbbZ_2}$.

Super-algebra structures on $2+e$ have been classified in \cite{aissaoui2013classification}. We wish to extend those to a Hopf algebra structure. Notice that this is different from the notion of a \textit{Hopf superalgebra} in \cite{aissaoui2013classification} which has a fermionic braiding, while the braiding here in $\Vec_G$ is trivial. Mathematica computations show that there is a unique Hopf algebra structure on $\mcH = 2+e = \mbbC^2 \oplus \mbbC$. As a super-algebra, it is identified with the $\mbbZ_2$-graded algebra $\mbbC[x,y]/ \langle x^2-x, y^2-x, xy -y \rangle$, where $1$ and $x$ have grading $0$ and $y$ has grading $1$. With this identification, other structural maps are given as follows:
\begin{equation}
\begin{aligned}
\Delta(x) &= 1 \otimes x + x \otimes 1 -\frac{3}{2} x \otimes x + \frac{1}{2} y \otimes y, \\
\Delta(y) &= 1 \otimes y -\frac{3}{2} x \otimes y + y \otimes 1 -\frac{3}{2} y \otimes x.
\end{aligned}
\end{equation}
\begin{align}
\epsilon(x) = \epsilon(y) = 0.
\end{align} 
\begin{align}
S(x) = x, \qquad S(y) = -y.
\end{align}

By definition, $\mcH$ is commutative. It also has an integral $1-x$ and $\epsilon(1-x) = 1$. As a graded algebra, $\mcH$ decomposes into the direct sum of two indecomposable graded subalgebras,
\begin{align}
\mcH = \Span\{1-x\} \oplus \Span\{x,y\},
\end{align}
or abstractly $2+e = \unit + (\unit + e)$. The algebra structures on $\unit$ and $\unit + e$ correspond to the trivial subgroup and $\mbbZ_2$, respectively. Therefore, by definition (\cite[Def. 9 and Remark 5]{ostrik2003module}), $\mcH$ is semisimple and there are three irreducible graded modules $A, B, C$.  Indeed, $\mcH$ decomposes to two algebras $\unit$ and $\unit+e$, therefore its module category decompose to the two module categories for the two algebras and we know already that these algebras give indeed semisimple abelian categories as their module category.

\begin{remark}\label{rmk3}
Here is an observation that will be used repeatedly later. A Hopf algebra $\mcH$ in $\mcB$ always has the trivial irreducible module $(\unit, \epsilon)$ and moreover for any simple object $h \in \mcB$, $(h, \epsilon_\mcH \otimes id_h)$ is also an irreducible module, where $\epsilon_\mcH \otimes id_h\colon \mcH \otimes h \to h$ is the module map.
\end{remark}

$A$ is the trivial module induced by the counit. $B$ is the 1-dimensional space with the non-zero part concentrated on degree $1$. The action of $\mcH$ on $B$ is also given by the counit. $C$ is the 2-dimensional space spanned by $x$ and $y$, and the action of $\mcH$ is given by multiplication. Since $1-x$ is an integral, $\Span\{1-x\}$ as an $\mcH$-module is isomorphic to the trivial irrep $A$, and hence the regular representation decomposes as $\mcH = A \oplus C$.
 
It is direct to see that the three modules $\{A, B,C\}$ form the category $Rep(S_3)$. For instance, the decomposition $C \otimes C = A + B + C$ is given as follows.
\begin{equation}
\begin{aligned}
A &= \Span\{x \otimes x - y \otimes y\}, \\
B &= \Span\{x \otimes y - y \otimes x\},\\
C &= \Span\{x \otimes x + y \otimes y, -x \otimes y - y \otimes x\},
\end{aligned}
\end{equation}
where the first term in $C = \Span\{\cdots\}$ corresponds to $x$ in $C$ and the second corresponds to $y$.

\begin{remark}
The Hopf structure on $2+e$ is an explicit example of a categorical Hopf algebra that has no analogue in Vec as the only semi-simple $3$-dimensional Hopf algebra is the group algebra $\mathbb{C}[\mathbb{Z}_3]$.
\end{remark}

\subsubsection{$\Vec_{\mbbZ_3} = \{\unit, e, e^2\}$}\label{5.1.3}

We would like to see how much the same pattern repeats itself. Inspired by the previous case where the Hopf algebra decomposed as an algebra, one can consider the direct sum of the two algebras $\mcH= \unit \oplus \Vec_{\mbbZ_3}$, where by abusing notation we use $\Vec_{\mbbZ_3}$ to denote the object $\unit \oplus e \oplus e^2$. As a graded algebra, $\mcH$ is identified with $\mbbC[y_1]/ \langle\, y_1^4 - y_1\,\rangle$ where $y_1$ has degree $1$. Hence a basis for $\mcH$ is given by $\{1, y_1^3, y_1, y_1^2\}$.  To do computations with Mathematica, one can use the fact that $\Delta(y_1)^i=\Delta(y_1^i)$. This greatly reduces the number of parameters in the computation. With this trick, all possible Hopf algebra extensions for $\unit \oplus \Vec_{\mbbZ_3}$ can be recovered. In fact, it turns out that there is only one, whose structure maps are as follows. The counit satisfies $\epsilon(1) = 1$, $\epsilon(y_1^i) = 0$, $i=1,2,3$. The antipode $S$ is the identity map unlike the previous case, and the comultiplication is given as follows:
\begin{equation}
    \begin{aligned}
\Delta(1) &= 1 \otimes 1,\\
\Delta(y_1^3)&= 1 \otimes y_1^3 + y_1^3 \otimes 1 -\frac{4}{3}y_1^3 \otimes y_1^3-\frac{1}{3}(y_1\otimes y_1^2+y_1^2\otimes y_1),\\
\Delta(y_1)&= 1 \otimes y_1 + y_1 \otimes 1+\frac{2}{3}y_1^2\otimes y_1^2 -\frac{4}{3}(y_1\otimes y_1^3+y_1^3\otimes y_1),\\ 
\Delta(y_1^2)&= 1 \otimes y_1^2 + y_1^2 \otimes 1+\frac{2}{3}y_1\otimes y_1 -\frac{4}{3}(y_1^2\otimes y_1^3+y_1^3\otimes y_1^2).
\end{aligned}
\end{equation}

Similar to the previous case, the decomposition as subalgebras $\mcH = \Span\{1-y_1^3\} \oplus  \Span\{y_1^3,y_1,y_1^{2}\}$ corresponds abstractly to the decomposition $\mcH = \unit \oplus  \Vec_{\mbbZ_3}$.

The irreducible modules of $\mcH$ are given by its algebra decomposition. The algebra $\Span\{1-y_1^3\}$ gives three modules from the action of counit denoted by $\{y_1^3,y_1,y_1^2\}$, and the second algebra $\Span\{y_1^3,y_1,y_1^2\}$ gives itself as a module, called $\rho$, with the action given by multiplication. The module category is therefore $\{y_1^3,y_1,y_1^2,\rho\}$. As $\{y_1^3,y_1,y_1^2\}$ forms a $\Vec_{\mbbZ_3}$ as a fusion category, it remains to understand how $\rho$ fuses with itself and other simple modules. This is all done by direct calculations and the fusion rules turn out to be a generalization of the previous case, giving the near-group category with the so-called multiplicity $m=|\mbbZ_3|-1=2$:
\begin{align}
\rho \otimes \rho &= y_1^3\oplus y_1 \oplus y_1^2 \oplus 2\rho \\ 
\rho \otimes \alpha &= \rho, \ \ \forall\alpha \in \{y_1^3,y_1,y_1^2\}
\end{align}
This is also noted \cite[Thm. 1.1]{larson2014pseudo} to be $Rep(A_4)$, the representation category of the alternating group $A_4$.

For a general prime $p$, we can also consider the algebra $\unit \oplus \VecZp$ in $\VecZp$ where the first $\VecZp$ denotes the direct sum of all simple objects, each with multiplicity one. The cases for $p=2$ and $p=3$ motivate the conjecture that the algebra $\unit\oplus \VecZp$ has a Hopf algebra structure with representation category the near group category given by $G=\mbbZ_p$ and multiplicity $m=|G|-1$. However, a classification by \cite[Theorem 5.1]{izumi2017cuntz} shows that can only happen if $p+1$ is a prime power. Therefore this conjecture is certainly not true in general and needs to be restricted:
\begin{conjecture}\label{cnj5.1}
In $\VecZp$, the algebra $\unit \oplus \VecZp$ admits a categorical Hopf algebra structure whose representation category is the near group category given by $G=\mbbZ_p$ and multiplicity $m=|G|-1$, if and only if $p=q^m-1$ for some prime $q$.
\end{conjecture}
Mathematica computations show that there is no Hopf algebra structure on the algebra $\unit\oplus \Vec_{\mbbZ_5}$. Therefore, one can not hope for a Hopf algebra on $\unit\oplus \VecZp$ for all $p$ either.

\subsection{Condensation Examples}\label{5.2}
We examine the condensation of some anyon theories which should give a large class of examples for extension by Hopf comonads. We will follow the notations defined in \hyperref[3.2.4]{3.2.4}. 
\subsubsection{$\text{Rep}(S_3)$}\label{5.2.1}
A not so complicated condensation is that of  $\mcB=\text{Rep}(S_3)=\{A,B,C\}$ with the condensable algebra $\mcA=A+C$. Notice that the fusion rules are $B\otimes B=A,\, C\otimes C=A+B+C,\,B\otimes C=C$ and $A$ is the unit. Further, to see that $\mcA=A+C$ is a condensable algebra, we note that $\Hom(\unit,\mcA)\cong \mbbC$ (connectedness) and the two conditions \cite[Thm. 3.7, 3.8]{cong2016topological} hold:
\begin{itemize}
    \item (algebra commutativity) all objects inside $\mcA$ have trivial twists,
    \item (separability) for all objects $a,b \in \mcC$, there is a partial isometry 
    $$\Hom(a,\mcA)\times \Hom(b,\mcA) \to \Hom(a \otimes b, \mcA).$$
\end{itemize}
In this and the next example, we follow the notations in \cite[Table 1]{cong2016topological}. The condensation yields the category $\mcB_\mcA=\Vec_{\mbbZ_2}$ where the deconfined part is itself. To prove this, we use the Frobenius reciprocity. While this example is done explicitly, a very similar process is used for all other examples. Frobenius reciprocity for pairs 
$$(X,Y)=\{(A,A),(A,B),(B,B),(A,C),(B,C),(C,C)\}$$
implies, respectively:
\begin{itemize}
    \item $D_\mcA(A)$ is a simple object as $\Hom_{\mcB_\mcA}(A,A)=1$ and we call it $\unit$,
    \item $D_\mcA(B)$ contains no copy of $\unit$,
    \item $D_\mcA(B)$ is actually a simple, we call it $e$,
    \item $D_\mcA(C)$ has a copy of $\unit$,
    \item $D_\mcA(C)$ has a copy of $e$,
    \item $D_\mcA(C)$ has only two simples as $\Hom_{\mcB_\mcA}(C,C)=2$, therefore it is $\unit+e$.
\end{itemize}
The fusion rules of $\unit,e$ can be easily discovered by using the Frobenius reciprocity for pairs like $(B\otimes B,A),(B\otimes B,B),$etc. The result is $\mcB_\mcA=\Vec_{\mbbZ_2}$.

Hence, the tensor functor is $D_\mcA(A)=\unit, D_\mcA(B)=e, D_\mcA(C)=\unit+e,$ and the forgetful functor can be computed from the condition of right adjoint and it is given by $E_\mcA(\unit)=A+C, E_\mcA(e)=B+C$. From these, it can be deduced $T_\mcA(X)=D_\mcA E_\mcA(X)=(2+e)\otimes X$. As discussed in \hyperref[3.2.4]{3.2.4}, $T_\mcA$, being compositions of adjunctions, is automatically a Hopf comonad by general nonsense and it inherits the self-dual HM structure coming from Hopf algebra structure (which is unique) on $2+e$ derived in \hyperref[5.1.2]{5.1.2}. 

\subsubsection{$D(S_3)$}\label{5.2.2}

Consider the case $\mcB=D(S_3)$ and the condensable algebra $\mcA=A+C$. The objects are denoted by $\{A,B,C,D,E,F,G,H\}$ where $\{A,B,C\}$ is the canonical image of $\text{Rep}(S_3)$ in $D(S_3)$. By using the framework in \hyperref[3.2.4]{3.2.4}, one derives the condensed category $\mcB_\mcA=\text{Rep}(D(\mbbZ_2)) \oplus \{X,Y\}$ with the following fusion rules for $X,Y$\footnote{The third author would like to thank C. Delaney for confirming the fusion rules of $X,Y$.}:
\begin{equation}
\begin{aligned}
mX=Y,mY=X,\psi Y=X,\psi X=Y,eY=Y,\\
X^2=\unit+e+Y,XY=m+\psi+X,Y^2=\unit+e+Y,
\end{aligned}
\end{equation}
where $\text{Rep}(D(\mbbZ_2))=\{\unit,e,m,\psi\}$ is the doubled $\mbbZ_2$ theory, $\{\unit,e\}$ is the image of the canonical embedding of $\text{Rep}(\mbbZ_2)=\Vec_{\mbbZ_2}$, and exchanging $m,e$ is a symmetry of the category. The maps $E_\mcA$ and $D_\mcA$ are given by
\begin{align}
E_\mcA: \unit \to A+C, e \to B+C, m \to D, \psi \to E, X \to D+E, Y \to F+G+H,\\
D_\mcA: A \to \unit , B \to e, C \to \unit+e, D \to m+X, E \to \psi+X,\  F,G,H \to Y.
\end{align}
Therefore, a straightforward calculation yields the Hopf comonad  $T_{\mcA}=D_\mcA E_\mcA$ to be acting by $(2+e) \otimes -$ on the objects $\unit,e,Y$ and $(\unit+Y) \otimes -$ on $m,\psi,X$. It is easy to see that   $\{X,Y\}$ can be replaced by $\{m+\psi,\unit+e\}$ as they are modules of $\unit+e$ and together with $\text{Rep}(D(\mbbZ_2))= D(\Vec_{\mbbZ_2})$ form the modules of $\unit \oplus (\unit+e)$ acting on $D(\Vec_{\mbbZ_2})$ as a HA with the structure in \hyperref[5.1.2]{5.1.2}.

\subsubsection{$D(\text{Ising})$ from the gauging of $SO(8)_1$ followed by a condensation}\label{5.2.3}
There is an $S_3$ symmetry on $SO(8)_1$ which can be gauged (as in \cite[6.2]{cui2016gauging}) giving the braided category $(SO(8)_1)_{S_3}^{\times,S_3}$ which has objects 
$$A,B,C,X, \ {}_{\alpha}X,\ {}_{\alpha^*}X,Y_+,Y_-,X^{++},X^{+-},X^{--},X^{-+}$$
with fusion rules described in the \hyperref[8.3]{Appendix 8.3}. One can condense $A+C$, where $A,B,C$ are the same as $A,B,C$ in $\text{Rep}(S_3)$. The fusion rules can be computed and it will give the double of Ising and three defects. The important observation to be made is that although condensation of a double is a double but the opposite is not true. We also have a change in the central charge which is not possible in the group gauging process.

\subsubsection{$\text{Haag}$}\label{5.2.4}
This example is about the famous Haagerup fusion category, denoted as $H_2$ in \cite{grossman2012quantum}, a UFC of rank 6 with three indecomposible module categories coming from \textit{Q-systems} $\unit, \unit+ \rho, \unit+\alpha+\alpha^*$. For Q-systems definition, see e.g. \cite[3.2]{bischoff2015tensor}. The corresponding bimodule fusion categories are denoted as $H_2, H_1, H_3$ in \cite{grossman2012quantum}. It is hoped that 
there is a category $\mcC$ which would make the sequence $\mcC \to H_3 \to H_2$ a gauging process where the second functor is the forgetful functor.

\subsubsection{$D(\text{Fib})$}\label{5.2.5}

This example is conjectured to be part of a sequence of examples which should build the sequence $D(\text{Haag}_k)$ defined later in \hyperref[6.3.3]{6.3.3}. for $k=0$ one recovers $D(\text{Fib})$. The condensable algebra $\mcA=\unit \unit + \tau \bar{\tau}$ gives the condensed category $\mcB_\mcA=\text{Fib}$, where $\unit$ is deconfined and $\tau$ is a defect. The action $T_\mcA=D_\mcA E_\mcA$ can be computed from
\begin{align}
E_\mcA: \unit \to \unit \unit + \tau\bar{\tau},\ \tau \to \unit\bar{\tau}+\bar{\tau}\unit+\tau\bar{\tau},\\
D_\mcA: \unit \unit \to \unit,\ \tau\bar{\tau} \to \unit+ \tau, \ \unit\bar{\tau}, \bar{\tau}\unit \to \tau
\end{align}
and is given by $(2+\tau) \otimes -$ which has the HA structure described in \hyperref[8.2]{Appendix 8.2}.

\section{Generalized Symmetries}\label{6}

In this section, we first define category symmetries and then outline a program to generalize the theory of extension and gauging to category symmetries.  We show that a category symmetry is essentially a tensor functor. 

We will work with module categories over a braided fusion category $\mcB$.  But the theory generalizes to fusion category simply by replacing module categories with bimodule categories.  In practice, we are really working with bimodule categories of $\mcB$, but only those coming from module categories using the braidings.  In the example of $\Vec_{\mbbZ_2}$, two out of the six bimodule categories are such bimodules.

\subsection{Category symmetry}\label{6.1}

Let $\mcF$ be a fusion category with a complete representative set of simple objects $\{X_i\}$, and $\mcB$ be a unitary braided fusion category.  The monoidal category of module categories $\Modc(\mcB)$ contains the Picard group $\Pic{\mcB}$, which is isomorphic to the group of braided tensor auto-equivalences $\Aut{\mcB}$ of $\mcB$.  $\Modc(\mcB)$  should be regarded as the group algebra of $\Pic{\mcB}$, and therefore the subcategory $\textrm{M}\Pic{\mcB}\subset \Endo(\mcB)$ generated by all functors from $\Modc(\mcB)$ should be regarded as the group algebra of $\Aut{\mcB}$.

The isomorphism between $\Aut{\mcB}$ and $\Pic{\mcB}$ underlies the physical relation between symmetries and defects.  A generalization of this isomorphism to HM symmetries will be a  characterization of the endo-functors M$\Pic{\mcB}$ that correspond to $\Modc(\mcB)$.  The existence of such a correspondence is obvious from the physical interpretation of unitary modular categories as anyon models.  Suppose $\mcM$ is an indecomposible module category over a unitary modular category $\mcB$, then $\mcM$ forms a gapped boundary between $\mcB$ and itself.  Forming the double $\mcB\boxtimes {\mcB}^{op}$ of $\mcB$, we can drag an anyon in the top layer $\mcB$ to the boundary, and then lift it to the bottom layer resulting a map $\theta_M: \mcB \rightarrow \mcB$.  The interesting mathematical question is to characterize all such resulting endo-functors M$\Pic{\mcB}$ of $\mcB$.

The theory of extension and gauging for group symmetries is a lifting problem. Starting with a group homomorphism $\rho: G \rightarrow \Aut{\mcB}$, we look for a lifting first to $\u{\rho}: \underline{G}\rightarrow \uPic{\mcB}$, and then $\uu{\rho}:\underline{\underline{G}}\rightarrow \uuPic{\mcB}$.  The generalization will be simply to replace $\Pic{\mcB}$ with M$\Pic{\mcB}$.
\begin{definition}\label{dfn6}
\begin{enumerate}
    \item A category symmetry of a fusion category $\mcC$ is a pair $(\mcF, \rho)$, where $\mcF$ is a fusion category and $\rho: \mcF \rightarrow \mcC$ a tensor functor.
   \item A Hopf symmetry of a fusion category $\mcC$ is a Hopf monad $T: \mcC\rightarrow \mcC$.
\end{enumerate}
\end{definition}

\begin{prop}\label{prop1}
A Hopf symmetry $T$ of a fusion category $\mcC$ gives a category symmetry $(\mcC^T, U)$ of $\mcC$, where $U\colon \mcC^T \to \mcC$ is the forgetful functor. Conversely, any category symmetry $(\mcF,\rho)$ arises in such as a way. 
\end{prop}
\begin{proof}
Obvious from definitions in \hyperref[3.1]{3.1} and the fact that tensor functors on fusion categories have adjoints (see e.g. \cite[1.3]{bruguieres2011exact}).
\end{proof}

\subsection{Extension}\label{6.2}
In this part, we provide evidence for the \textit{existence} of an extension theory based on Hopf symmetry, a generalization of the one based on group symmetry. Condensation as in \hyperref[5.2]{5.2} would also be contained as examples of this theory. This part is conjectural in nature and does not follow the theorem-proof format.

First we see how the group symmetry case can be recovered in the language of HMs. Most importantly for $T_G=\oplus T_g :\mcC \to \mcC$, why does the extension have a $G$-grading? The $G$-grading is the same as a $\Vec_G$ grading. Therefore, it is important to understand how to \textit{recover} $\Vec_G$ from $T_G$. 

The obvious solution seems to look at the modules of $T_G$, but this gives the equivariantization $\mcC^G$ which, in spite of having $Rep(G)$ inside, does not contain $\Vec_G$ canonically. Most importantly, we would like something that has itself a grading ``\textit{equal}'' to $\Vec_G$.

The answer turns out to be the $T_G(\unit)-$\textit{co}modules. Indeed, just like any HM, $T_G(\unit)$ happens to be a coalgebra with the coalgebra structure $(T(\unit),T_2(\unit,\unit), T_0)$ \cite{bruguieres2007hopf}. As $T_G(\unit)=\oplus T_g(\unit)=\oplus \unit_g$ and the coproduct $T_2(\unit,\unit)\colon T_G(\unit) \to T_G(\unit) \otimes T_G(\unit)$ can be decomposed to $\unit_g \to \unit_g \otimes \unit_g$, we have $\mcC^{co-T_G(\unit)}=\oplus_g \ \mcC^{co-T_g(\unit)}$.

Therefore, for a general HM $T$, a good starting point would be to consider the decomposition of $T(\unit)$ to simple coalgebras $T(\unit)=\oplus m_i$ for the grading of the extension 
$$\mcC^\times_T=\oplus \mcC_{m_i}.$$
For a Hopf comonad, one has to consider the decomposition of $T(\unit)$ to simple algebras as $T(\unit)$ is an algebra instead of a coalgebra. 

Now we examine what each component $\mcC_{m_i}$ should be. For the case of $T_G$, the $m_i\,'$s correspond to elements of $G$, and $\mcC_{m_i} = \mcC_{g}$ is known to be formed by objects of $\mcC$ fixed by $T_g$. Since $m_i = T_g(\unit)\cong 1$, each $m_i$-comodule is simple a copy of $\mcC$. Hence, we need to impose more restrictions on the objects that are allowed in the extension.  A plausible generalization would be to somehow seek for a decomposition of the HM $T$ into a sum of functors $T_{m_i}$, and require $\mcC_{m_i}$ to consists of objects $X$ of $\mcC$  that are not only $m_i$-comodules but also satisfy $T_{m_i}(X) \cong m_i \otimes X$.

Let us see how this shapes the story for group and Hopf algebra symmetry. Notice that for groups, $m_i=\unit_g$ and by construction we have the decomposition $T = \oplus T_g$. Thus the condition of being an $m_i$-comodule is trivial, while the nontrivial condition is that of $T_g(X) \cong \unit_g\otimes X = X$, which  exactly means $X$ is fixed by $T_g$. 

This is in contrast to $T(-)=\mcH \otimes -$ for a Hopf algebra $\mcH$ where $T(\unit)=\mcH=\oplus m_i$ decomposes to simple coalgebras $m_i$. One can obviously guess what $T_{m_i}\,'$s should be; they are defined as $T_{m_i}(-)=m_i\otimes -$ and $T = \oplus T_{m_i}$.   Therefore the condition $T_{m_i}(X)\cong m_i\otimes X$ is trivial, while that being an $m_i$-comodule is not. Hence for the HM $T(-)=\mcH \otimes -$, the extension $\mcC^\times_T$ should be essentially given by $\mcH$ comodules. 

Notice that for the Hopf comonads such as the ones found in the condensation examples, one needs to look at decompositions of $T(\unit)$ into simple algebras. As the HM $T=\mcH\otimes -$ is a  Hopf comonad as well, one can as well consider algebra decomposition of $\mcH$ as the grading for another extension. More generally, for self-dual HMs there are two ways of considering the extension. 

For example, the Hopf comonad $T=(2+e)\otimes - $, as demonstrated in \hyperref[5.1.2]{5.1.2}, has the decomposition $\unit \oplus (\unit+e)$, and as mentioned in \hyperref[5.2.2]{5.2.2}, this gives the extension given by $\mcH=2+e$ modules:
$$\mcC^\times_T = \mcC_\unit \oplus \mcC_{\unit+e} = D(\mbbZ_2)\oplus \{X,Y\}.$$

How do we generalize the story to all HMs $T$? First, one notes that $(T(X),T_2(\unit,X))$ is always a $T(\unit)$-comodule \cite{bruguieres2007hopf}. Therefore, it must decompose to $m_i$ comodules. One can then define $T(X)= \oplus T_{m_i}(X)$ as being the decomposition to $m_i$ comodules. Therefore, there is a definition for $T_{m_i}$: a certain functor from $\mcC$ to $\mcC^{co-m_i}$ given by the composition of the functor $U:\mcC \to \mcC^{co-T(\unit)}$ where $U(X)=(T(X),T_2(\unit,X))$, followed by the projection onto $\mcC^{co-m_i}$. 

The next step is a crucial observation derived from the identities of the fusion operator $H^l_{\unit,X}:T^2(X) \to T(\unit)\otimes T(X)$. This is an isomorphism by the definition of a HM, and further the fusion operator is known to satisfy \cite[Prop. 2.6]{bruguieres2011hopf}:
$$(T_2(X,Y) \otimes id_{T(Z)})H^l_{X\otimes Y,Z}=(id_{T(X)}\otimes H^l_{Y,Z})T_2(X,Y\otimes T(Z)).$$
Replacing $X,Y$ by $\unit$ and $Z$ by $X$, implies that $H^l_{\unit,X}$ is a $T(\unit)$-comodule map between $(T^2(X),T_2(\unit,T(X))$ and $(T(\unit)\otimes T(X),T_2(\unit,\unit)\otimes id_{T(X)})$. Therefore, as $T(\unit)$ decomposes, we obtain an isomorphism between their decomposition into $m_i$-comodules given by:
$$T_{m_i}(T(X)) \xrightarrow{\cong} m_i \otimes T(X).$$
The above is \textit{always} true, and it shows that the idea of fixed point generalization as $T_{m_i}(Y) \cong m_i \otimes Y$ makes sense, as all $Y=T(X)$ satisfy this property. This can be also checked for $T=T_G$ where indeed $T_g$ fixes any $T_G(X)$.

Notice the above discussion was for all \textit{left} $T(\unit)$-comodules, and the same can be said about right $T(\unit)$-comodules by using the right fusion operator.

Treating $\unit$ as a coalgebra with the apparent coalgebra structure, the map $\eta_{\unit}\colon \unit \to T(\unit)$ becomes an injective coalgebra morphism by the compatibility conditions between $T_2,\ T_0$, and $\eta$ \cite{bruguieres2007hopf} (replacing $X, Y$ with $\unit$):
\begin{align}
    \label{equ:T2_eta_compatibility}
    T_2(X,Y) \eta_{X\otimes Y}=\eta_X \otimes \eta_Y \quad \text{and} \quad T_0 \eta_{\unit} = 1.
\end{align}
This means $\unit$ is a simple subcoalgebra inside $T(\unit)$, implying that in the decomposition of $T(\unit)$ into simple coalgebras one of the $m_i\,'$s can be chosen to be $\unit$. We {\it identify} $\unit$ with its image in $T(\unit)$ under $\eta_{\unit}$.  It is conjectured (see the next paragraph) that the associated functor $T_{\unit}$ is the identity functor, in which case we have $\mcC_{\unit} = \mcC$ and therefore  $\mcC^\times_T$ is an \textit{extension} of $\mcC$. 

We claim that $T_{\unit}(X)$ contains $X$. First of all, any $X$ is trivially a comodule of $\unit$. The embedding $\eta_{\unit}: \unit \to T(\unit)$ induces a $T(\unit)$-comodule structure on $X$, where the comodule map is given by $\eta_{\unit} \otimes id_{X}$. Then it is direct to see that the map $\eta_X: X \to T(X)$ is a $T(\unit)$-comodule morphism by the identity $T_2(\unit, X) \eta_X = \eta_{\unit} \otimes \eta_{X}$ which is obtained from the first identity in Eq. \ref{equ:T2_eta_compatibility} with $X, Y$ replaced by $\unit, X$, respectively. $\eta_X$ is clearly also injective. Hence, $X$ is a sub comodule of $T(X)$. Moreover, by the same identity as above, $T_2(\unit, X)$ sends the image of $X$ to $\unit \otimes T(X) \subset T(\unit) \otimes T(X)$. But $\unit \otimes T(X)$ is also where $T_{\unit}(X)$ is exclusively sent by $T_2(\unit,X)$ due to the comodule decomposition. Hence $X \subset T_{\unit}(X)$ and it is conjectured to be equal.

One important aspect in the extension theory concerns the fusion rules. For the group case, the fusion rules correspond to possible liftings of an action $G \to \text{Aut}_\otimes(\mcC)$ to $\underline{G} \to \underline{\text{Aut}_\otimes(\mcC)}$. Notice that a lifting is presumed from the \textit{beginning} as one states that $T_G$ is a HM coming from  $T_{-}: \underline{G} \to \underline{\text{Aut}_\otimes(\mcC)}$. Therefore it will not be surprising to see that there is always a fusion rule for $\mcC^\times_T$ although there might be more than one possible fusion rule just like in the group case. The Hopf algebra case should be similar, where the extension is given by the $T$-comodule category which a fusion category as $T$ is also a Hopf comonad.

As for the general case, the fusion $\mcC_{m_i}\otimes \mcC_{m_j}$ is likely to be given by the multiplication     $\mu:T^2 \to T$ whose restriction to $T_{m_i}(T_{m_j}(X))$ should give an indication of how $\mcC_{m_i}$ and $\mcC_{m_j}$ must fuse. Putting it more abstractly, if $T(\unit)$ happens to be a bialgebra and hence has a monoidal comodule category, one seeks a monoidal functor from the category of $T(\unit)$-comodules to the monoidal category $\Bimodc(\mcC)$, as all $\mcC_{m_i}$ are $\mcC$-bimodules by fusion from left and right.
 
As for the associator, in the group case it is given by some vanishing obstruction in the fourth cohomology and classified by the third cohomology $H^3(G,U(1))$, and there is likely a similar story for the general case. It is known that deformations of the associator are given by the third Yetter-Davydov cohomology (see e.g. \cite[Chap. 7.22]{etingof2016tensor}). 

Applying this reasoning to an extension of $\Vec$ given by an ordinary Hopf algebra $\mcH$, we should be looking at $H^i(\mcH\text{-}comod)$ which for cocommutative hopf algebras turn out to be $HH^i(\mcH,\mbbC)$, the Hochschild cohomology of $\mcH$. As an example, notice that the trivial group symmetry in the case of $\Vec$ coincides with that of the action of the cocommutative Hopf algebra $\mbbC[G]$. Hence, taking $\mcH=\mbbC[G]$, this results in the cohomology $HH^i(\mbbC[G],\mbbC)$ which is known to be the cohomology of $G$ (see e.g. \cite[Chap. 7.22]{etingof2016tensor}), making the story compatible with the group extension. 

A final piece of the puzzle is the extension of $T$ to a HM $T^\times$ on $\mcC^{\times}$. For the group case, it is known to be unique \cite{etingof2009fusion}, while it might not be so for the Hopf algebra case. 

Indeed, as mentioned before, for a Hopf algebra action $T(-)=\mcH\otimes -$ as a HM (Hopf comonad), the extension $\mcC^\times_T$ is nothing but the $T$-comodule(-module) category itself. Generally for a self dual HM $T$, a canonical extension of $T$ to its module category $\mcC^{T}$ is always possible by $T^\times : (M,r) \to (T(M),T(r))$. This can be easily shown to be a self-dual HM and also obviously an \textit{extension} of $T$ acting on the copy of $\mcC$ inside $\mcC^{T}$ given by modules $(X,\epsilon_X)$ with $\epsilon$ the counit in \hyperref[dfn3]{\textbf{Definition 3}}. 

Now consider the condensation process in \hyperref[5.2.2]{5.2.2}, where we switch to Hopf \textit{co}monad $T_\mcA$ acting on $\mcB_\mcA$. From its appearance, $T_\mcA$ does not seem to be an extension of $T=(2+e)\otimes (-)$ until one considers the replacement of $X,Y$ by $m+\psi,\unit+e$. But $T_\mcA$ is certainly not the canonical extension given above. Indeed for the canonical extension $T^\times$ acting on module $(M,r)=(m,\epsilon_m)$, one obtains $T^\times((M,r))=(2m+\psi,\epsilon_{2m+\psi})$ which is not the result of $T_\mcA$ acting on $m$ by $\unit+Y$ even at the object level. In fact, it is not hard to prove that $T^\times=(2+e)\otimes$. This shows the possibility of different extensions of $T$ itself. 

There are many questions for which there are hints for the answers and are yet to be explored. We will explore a general HM extension theory in future works.

\subsection{Gauging applications} \label{6.3}

\subsubsection{A generalization of Fib to all primes $\Fib_p$}\label{6.3.1}

Regarding $\Fib$ as a fusion category, we consider its generalization to a sequence of near-group categories \cite{siehler2003near} whose fusion rules would be denoted as $\Fib_p$ for all primes $p$, with $p=1$ understood as Fib.

There are $p+1$ isomorphic classes of simple objects denoted as $\mbbZ_p$ and $X$ with the non-group fusion rule:

$$g\otimes X=X\otimes g=X, \quad  X\otimes X=   (\bigoplus\limits_{g \in \mbbZ_p} g) \oplus p\, X.$$

For $p=2$, this is the well-studied fusion category $\frac{1}{2}E_6$, and for $p=3$, this is the Izumi-Xu fusion rule; see also \cite[p. 589]{evans2014near} for more examples and a classification of near-group categories.

\begin{opq}\label{opq6.1}

Is there a HM on $\VecZp$ whose extension realizes $\Fib_p$ for each prime $p$?
\end{opq}

\subsubsection{A generalization of the Haagerup category $\Haag_p$}\label{6.3.2}

Fib as a fusion category fits into another potential sequence of fusion categories whose fusion rules will be denoted as $\Haag_p$.  $\Haag_p$ has $2p$ classes of simple objects denotes as $\alpha^i, i=0,1,...,p-1,$ and $\rho_i, i=0,1,..,p-1$, where $\alpha^0=\unit, \rho_0=\rho$. The $\alpha^i\,'$s obey $\mbbZ_p$ fusion rule. The non-group fusion rules are determined by: 

$$\alpha^i\otimes \rho=\rho_i=\rho \otimes \alpha^{p-i}, \quad \rho^2=\unit \oplus \sum_{i=0}^{p-1} \rho_i.$$

For $p=2$, this is the fusion rule of $\textrm{PSU}(2)_6$, and for $p=3$, this is the Haagerup fusion rule.

\begin{opq}\label{opq6.2}
Is there a HM on $\VecZp$ with an extension that realizes $\Haag_p$ for each prime $p$?
\end{opq}

\subsubsection{The Doubled Haagerup category $D(\text{Haag}_p)$}\label{6.3.3}
The hypothetical modular category $D(\text{Haag}_p)$, defined in \cite{evans2011exoticness}, for all odd prime $p$ is of rank $p^2+3$ with anyons $\unit,b, a_h, d_l$ with $1\leq h \leq \frac{p^2-1}{2}, 1\leq l \leq \frac{p^2+3}{2}$ of quantum dimensions $1, p\delta +1, p\delta+2, p\delta$, and $\delta=\frac{p+\sqrt{p^2+4}}{2}$ satisfying $\delta^2=1+p\delta$.
It is known to exist for $p\leq 13$ \cite{evans2014near}.  

\begin{prop}\label{prop2}
The object $\mcA=\unit+b$ of DHaag has a condensable algebra structure.
\end{prop}

\begin{proof}
DHaag is weakly monoidal Morita equivalent to $\textrm{Haag}\boxtimes \textrm{Haag}^{op}$ \cite{mueger03},
and $\unit+\rho \boxtimes \rho^*$ is a condensable algebra, which is sent to $\unit+b$.
\end{proof}

Consider the condensable algebra $\mcA=\unit+b$, then we have\footnote{The third author thanks T. Gannon for computing the fusion rules of the condensed category of DHaag.} 
$$D_\mcA(\unit)=\unit, D_\mcA(b)=\unit+X, D_\mcA(a_h)=X+\alpha_{i,j}+\alpha_{i,j}^*, D_\mcA(d_l)=X,$$
where $\alpha_{i,j}$ form $D(\mbbZ_p)$, and $X^2=D(\mbbZ_p)+ p^2 X$.  So $\mcC=D(\mbbZ_p) \oplus\{X\}$ with deconfined $\mcD=D(\mbbZ_p)$.  
It follows the HM is 
$$T_\mcA(\unit)=2+X, T_\mcA(\alpha_{ij})=\alpha_{i,j}+\alpha_{i,j}^*+X, T_\mcA(X)=D(\mbbZ_p)+ (p^2+2)X.$$
So $T_\mcA(a)=a+a^*+X\otimes a$ is a HM on $D(\mbbZ_p)\oplus \{X$\}, and the question in general is what $T$ is on $D(\mbbZ_p$)?

\begin{opq}\label{opq6.3}
Can $D(\text{Haag}_p)$ be realized through gauging a Hopf monad symmetry on $D(\mbbZ_p)$? 
\end{opq}

\subsection{The tensor product of modules of $D(\mbbZ_N)$}\label{6.4}
An important topic to study for Hopf monad extension is that of the monoidal category $\mcC^\times_T=\oplus \mcC_{m_i}$ formed by the bimodules $\mcC_{m_i}$. As a large class of examples involve the double of abelian groups $\mbbZ_N$, we have decided to include a study of the product of the modules of $D(\mbbZ_N)$ when viewed as a bimodule using the braiding structure. It is important to note that most of the work is already done in \cite[Proposition 3.19]{etingof2009fusion}, where one characterizes the bimodules of $\Vec_A$ for abelian group $A$ and computes their tensor product over $\Vec_A$. Indeed, in this case $A=\mbbZ_N \times \mbbZ_N$, and we only need to understand which bimodule correspond to the module given a bimodule structure by braiding. Although in this section, \cite[Proposition 3.19]{etingof2009fusion} is used directly without any explanation, in \hyperref[8.1]{Appendix 8.1}, we form a table of $\Vec_{\mbbZ_2}$-bimodule monoidal category by going explicitly through the construction of the fusion of bimodules and give more details of the calculations in this section.

We refer to the discussion at the beginning of \hyperref[5]{section 5}. Consider a module $\mcM=\mcM(A,\rho)$ of $\mcC=D(\mbbZ_N)$ coming from the algebra $A \subset \mbbZ_N \times \mbbZ_N$ with a 2-cocycle $\rho \in H^2(A,\mbbC)$. Notice for abelian groups 2-cocycles and skew-symmetric bicharacters form the same group (see e.g. \cite[p.12]{etingof2009fusion}). We will assume that $\rho$ is trivial to simplify calculations but it is easy to do the general case as well. The right action of $\mcC=D\mbbZ_N$ on $\mcM$ is given by
$M . X := X . M$, i.e. just the left action, which is actually fusion as $\mcM$ is really inside $\mcC$. Recall the associativity constraint for the left action denoted by
$$m_{X,Y,M}:= (X\otimes Y) . M \to X.(Y.M).$$
The braiding comes in the associativity constraints for the right action, given by:
$$n_{M,X,Y} := M . (X \otimes Y) \to (M . X) . Y$$
and defined as $c_{X,Y}$ as the left side is $(X \otimes Y).M$ and the right side is $Y.(X.M)$. Of course, we are using the fact that the left action is given by fusion and the associator $\alpha$ of $\mcC$, which is the associator $m_{X,Y,M}$ of the left action, are all trivial. Otherwise, for the general case of braided $\mcC$, the right action is still defined as $M.X := X.M$, but with associativity constraint:
$$n_{M,X,Y} = m_{Y,X,M}(c_{X,Y} \otimes id_M).$$
The final structure map, the \textit{middle} associativity constraint
$$b_{X,M,Z} : (X.M).Z \to X.(M.Z)$$
is $m_{X,Z,M}c_{Z,X}m^{-1}_{Z,X,M}$ for the general case of a braided $\mcC$. The reason that these structure maps $(n)b$ satisfy the commuting diagrams constraints for the definition of (right module)bimodule \cite[Def. 7.1.7]{etingof2016tensor} can be derived by using the similar diagrams for the definition of the braided category $\mcC$ and the $\mcM$ as a module of $\mcC$. Although in the case of $D\mbbZ_N$, one can see that $m$ is the associator $\alpha$ of $\mcC$, and all diagrams are formed by braiding and the associator which commute due to the axioms satisfied by $\mcC$ itself.

By standard module category facts, $\mcM$ as a bimodule of $D\mbbZ_N$, is equivalent to a left module of $D\mbbZ_N \boxtimes D\mbbZ_N^{op}$. The associator for this action is given by:
\begin{align}
l_{X_1\boxtimes Y_1, X_2\boxtimes Y_2,M}&:=((X_1 \otimes X_2) \boxtimes (Y_1 \otimes_{op} Y_2)).M \to (X_1 \boxtimes Y_1) .((X_2 \boxtimes Y_2).M) 
\end{align}
i.e, 
\begin{align}
&((X_1 \otimes X_2).M).(Y_2 \otimes Y_1) \to (X_1.((X_2.M).Y_2).Y_1 \\ 
&l_{X_1\boxtimes Y_1, X_2\boxtimes Y_2,M}=b_{X_1,X_2.M,Y_2}n_{X_1.(X_2.M),Y_2,Y_1}m_{X_1,X_2,M}
\end{align}
In our case this gives $c_{Y_2,X_1}c_{Y_2,Y_1}$. The module $\mcM$ is given by the algebra $A$, and we need to see which algebra $B$ in $D(\mbbZ_N) \boxtimes D(\mbbZ_N)=\Vec_{\mbbZ_N^2 \oplus \mbbZ_N^2}$ (as a fusion category, which is why the superscript ``$op$'' is omitted), and which 2-cocycle $\psi$ on $B$ would give us the module on $\Vec_{\mbbZ_N^2 \oplus \mbbZ_N^2}$ corresponding to $\mcM$ as a bimodule on $D(\mbbZ_N)$. Notice that the fusion category structure is what determines the category of (bi)modules, and that is why one can ignore braiding and omit ``$op$''. The algebra $B$ is essentially the stabilizer of a simple object like $[A]$, the trivial coset, inside $\mcM$ (\cite[2.7]{etingof2009fusion}):
\begin{align}
(X \otimes Y).[A] = (X.[A]).Y=Y.(X.[A])= y+x+[A]
\end{align}
where the summation on the right is by looking at what elements X,Y correspond to in the abelian group. To have the coset fixed, it is obvious that one needs to have $x+y \in A$. Therefore, the algebra $B$ is given by all those pairs $(x,y) \in \mbbZ_N^2\oplus \mbbZ_N^2$ with $x+y \in A$. The 2-cocyle is derived from the associator by restricting to $B$ (as in \cite[2.7]{etingof2009fusion}). Therefore $\psi((x_1,y_1),(x_2,y_2))=c_{Y_2,X_1}c_{Y_2,Y_1}=c_{Y_2,X_1 \otimes Y_1}$. This is indeed a bicharacter on $B$ as $c_{a,b}=e^{\frac{2\pi i}{N}a_2b_1}$ where $a=(a_1,a_2) \in \mbbZ_N^2$ and similarly for $b=(b_1,b_2)$. Hence, 
$$\psi((x_1,y_1),(x_2,y_2))=e^{\frac{2\pi i}{N}y_{22}(x_{11}+y_{11})}.$$
The next step is to take this 2-cocycle and make it into a \textit{skew-symmetric} bicharacter, as only then one can use the theorem necessary to find the result of the bimodule product. Adding a suitable 2-coboundary $d\epsilon$ makes this change possible, where $\epsilon(g)=\psi(g,g)^{\frac{1}{2}}$ giving us
$$(d\epsilon)(g_1,g_2)=\psi(g_1,g_1)^{\frac{1}{2}}\psi(g_2,g_2)^{\frac{1}{2}}\psi(g_1+g_2,g_1+g_2)^{-\frac{1}{2}},$$
notice in general that any 2-cocycle $\omega$ (for abelian groups) are bicharacters which implies $\omega(2g,2g)=\omega(g,g)^4$ and the above calculation can be done for any $\omega$ to make it a skew-symmetric bicharacter. The new skew-symmetric bicharacter is
$$\psi(g_1,g_2)d\epsilon(g_1,g_2)=\psi(g_1,g_2)\psi(g_1,g_1)^{\frac{1}{2}}\psi(g_2,g_2)^{\frac{1}{2}}\psi(g_1+g_2,g_1+g_2)^{-\frac{1}{2}}$$
$$=\psi(g_1,g_2)^\frac{1}{2}\psi(g_2,g_1)^{-\frac{1}{2}}$$
which, by abuse of notation,
$$\psi((x_1,y_1),(x_2,y_2))=e^{\frac{2\pi i}{2N}(y_{22}(x_{11}+y_{11})-(x_{21}+y_{21})y_{12})}.$$
If the bicharacter $\rho$ on $A$ was not trivial, then one would have to include in the associator $l$ the product $\rho(X_1,X_2)\rho(Y_1,Y_2)$ and the rest would be similar. 

As an application, consider the module $\mcM$ on $D(\mbbZ_2)$ given by the algebra $A=0\times \mbbZ_2$ and the trivial bicharacter. Then $\mcM$ as a bimodule on $D(\mbbZ_2)$ is a module on $\Vec_{\mbbZ_2^2 \oplus \mbbZ_2^2}$ with the corresponding algebra $B=\{(x,y)| x+y \in 0\times \mbbZ_2\}$ which is 
\begin{align}
(0\times \mathbb{Z}_2,0\times \mathbb{Z}_2) \cup (1 \times \mathbb{Z}_2,1 \times \mathbb{Z}_2).
\end{align}
Further the bicharacter $\psi$ is also trivial as $x+y \in 0\times \mbbZ_2$ implies $x_{11}+y_{11}=x_{21}+y_{21}=0$. These together with \cite[Proposition 3.19]{etingof2009fusion}, imply
$$\mcM \boxtimes_{D(\mbbZ_2)} \mcM=2\mcM.$$
The fusion structure of bimodules is not necessarily \textit{symmetric} with respect to the algebras. For example, for $A=  \mbbZ_2 \times 0$ giving a bimodule $\mcM$,
$$\mcM \boxtimes_{D(\mbbZ_2)} \mcM=2\mcM((\mathbb{Z}_2 \times 0,\mathbb{Z}_2 \times 0),\psi_{trivial}).$$
where the product gives the bimodule corresponding to the algebra $(\mathbb{Z}_2 \times 0,\mathbb{Z}_2 \times 0) \subset \mbbZ_2^2 \oplus \mbbZ_2^2$ with the trivial bicharacter. Details are in the \hyperref[8.1]{Appendix 8.1}.

\subsection{Applications to physics}\label{6.5}

The relation between symmetry and defect in the group symmetry case should be extended to the generalized symmetry and defect case.  We will leave such a physics theory as \cite{BBCW14,cui2016gauging} to the future.

\section{Modular categories as a tetra-category}\label{7}

Regarding gauging HMs as a new method to construct modular categories, we use gauging to organize modular categories for a structure theory with an eye towards a classification.

\subsection{Weak Hopf monad}\label{7.1}

From the definition of category symmetry, a fusion category $\mcF$ is a symmetry of $\mcC$, then $\mcC$ is also a module category of $\mcF$ by using the tensor functor to induce the action: $Y . X := \rho(Y)\otimes X, Y \in \mcF, X \in \mcC$. Since $\Vec$ is not a Fib module category, it follows that Fib cannot be a symmetry of $\Vec$.  Therefore, we cannot regard $D(\text{Fib})$ as a gauging of a Fib symmetry on $\Vec$.

There is a generalization of HMs to weak Hopf monads \cite{bohm2011weak}. Then for a weak Hopf algebra $\mcH$, the functor $T_\mcH$ is a weak Hopf monad on $\Vec$ and its modules can realize \textit{any} fusion theory.  Hence, a weaker notion of category symmetry based on weak Hopf monads could make Fib to be a category symmetry of $\Vec$ as it is expected in example \hyperref[5.2.5]{5.2.5}. We will leave such a theory to the future and use such a possibility to speculate on a structure theory of modular categories. 

\begin{definition}\label{dfn7}
\begin{enumerate}
    \item A weak category symmetry of a fusion category $\mcC$ is a pair $(\mcF, \rho)$, where $\mcF$ is a fusion category and $\rho: \mcF \rightarrow \mcC$ a strong Frobenius functor.
   \item A weak Hopf symmetry of a fusion category $\mcC$ is a weak Hopf monad $T: \mcC\rightarrow \mcC$.
\end{enumerate}
\end{definition}

\subsection{Simple modular category}\label{7.2}

\begin{definition}\label{dfn8}

A modular category $\mcB$ is simple if $\mcB$ is prime (not a Deligne tensor product) and contains no non-trivial normal (or condensable) algebras.
\end{definition}

\subsubsection{Simple pointed categories}\label{7.2.1}
Classification of all simple MCs is certainly a hard problem. We could start to find all simple pointed MCs. These are coming from abelian groups $G$ with some non-degenerate quadratic forms $q: G \to \mathbb{R}/\mathbb{Z}$ which are classified in \cite[Thm. 5.4]{galindo2016solutions} where $q$ satisfies $\theta_a=e^{2\pi i q(a)}, \ \forall a \in G$. One needs to first look for those $q$'s which do not have a kernel, hence as $q(a)=c(a,a)$ (following the notations of \cite{galindo2016solutions}) we must have $c(a,a) \neq 0, \forall 0 \neq a \in G$. Those would be the $(G,q)$ that can not have a commutative algebra and therefore, a condensable algebra. 
This is easy following the classification result for the prime abelian anyonic theory. First we list all classes here as presented in \cite[Thm. 5.4]{galindo2016solutions}:
\begin{enumerate}
    \item If $p\neq 2$ and $\epsilon=\pm 1$, 
    $\omega_{p,k}^\epsilon$ denotes the abelian anyon with fusion rules given by $\mathbb{Z}/p^k\mathbb{Z}$ and abelian 3-cocycle 
    $(0,c)$, where $c(x,y)=\frac{uxy}{p^k}$, for some  $u\in \mathbb{Z}^{>0}$ with $(p,u)=1$ and $\left(\frac{2u}{p}\right)=\epsilon$, where $\left(\frac{}{p}\right)$ is the Legendre symbol.
    \item If $\epsilon\in (\mathbb{Z}/8\mathbb{Z})^{\times}$, 
    $\omega_{2,k}^\epsilon$ denotes the abelian anyon with fusion rules given by $\mathbb{Z}/2^k\mathbb{Z}$ and abelian 3-cocycle 
    
    \begin{align*} c(x,y)=\frac{uxy}{2^{k+1}}, \  \  \     \  \  \omega(x,y,z)=  \begin{cases} \frac{1}{2}, \qquad &\text{if } x=y=z=2^{k-1},\\
    0. \quad  &\text{ otherwise.}
    \end{cases}
    \end{align*}
    for some  $u\in \mathbb{Z}^{>0}$ with $ u \equiv \epsilon \Mod{8}$. The abelian anyons $w_{2,k}^1$ and $w_{2,k}^{-1}$ are defined for all $k\geq 1$ and $w_{2,k}^{5}$ and $w_{2,k}^{-5}$ for all $k\geq 2$.
    
    \item $E_k$ denoted the abelian anyon with fusion rules given by $\mathbb{Z}/2^k\mathbb{Z}\oplus \mathbb{Z}/2^k\mathbb{Z} $ and abelian 3-cocycle $(0,c)$, where $c\in \Hom(\mathbb{Z}/2^k\mathbb{Z}\oplus \mathbb{Z}/2^k\mathbb{Z},\mbbR/\mbbZ)$ is defined by
    
    \begin{align*}
    c(\vec{e}_i,\vec{e}_j)=\begin{cases} 0, \qquad &\text{if } i=j,\\
    \frac{1}{2^k}, \quad  &\text{if } i\neq j. 
    \end{cases}
    \end{align*}
    \item $F_k$ denoted the abelian anyon with fusion rules given by $\mathbb{Z}/2^k\mathbb{Z}\oplus \mathbb{Z}/2^k\mathbb{Z} $ and abelian 3-cocycle $(0,c)$, where $c\in \Hom(\mathbb{Z}/2^k\mathbb{Z}\oplus \mathbb{Z}/2^k\mathbb{Z},\mbbR/\mbbZ)$ is defined by
    
    \begin{align*}
    c(\vec{e}_i,\vec{e}_j)=\begin{cases} \frac{1}{2^{k-1}}, \qquad &\text{if } i=j,\\
    \frac{1}{2^k}, \quad  &\text{if } i\neq j. 
    \end{cases}
    \end{align*}
\end{enumerate}
We note that for the first class $\sum_{j=0}^{p-1} jp^{k-1}$ can be easily checked to be a nontrivial condensable algebra if $k>1$. For $k=1$ there are no trivial twists except for the unit object hence this gives simple modular categories. The second class is similar to the first in this regard except that $k$ needs to be greater than two to have trivial twists and a condensable algebra $\sum_{j=0}^{p-1} jp^{k-1}$.

The third class has always the condensable algebra $\sum_{j \in E_k} j$ as all twists are trivial. This is in contrast with the fourth class where all $F_k$s are simple modular categories without any trivial twist.
\subsubsection{$SU(2)_k$}\label{7.2.2}
Another set of examples would be to find the simple modular categories among $\mathfrak{su}(2)_k$. The simples are labelled by the common notation $0,\ldots,k$. Here we know that the only possible normal algebra is $0+k$ as $\theta_a=e^{\frac{2\pi i a}{4}}$ and that can only be $1$ when $k \equiv 0 \pmod 4$. Otherwise for $k$ odd, the category splits to two MCs $semi \times PSU(2)_k$ where $semi$ is precisely $0,k$ which by itself form an MC. The case which remains to prove that does not split to a product and is therefore simple, is $k\equiv 2 \pmod 4$.  This follows from the ADE classification of modular invariants of $SU(2)_k$ (e.g. see \cite{KO02,ostrik2002module}).

\begin{opq}\label{opq7.1}

\begin{enumerate}
\item Is it true that simple modular categories are  determined by modular data? Note that this is false for general modular categories \cite{MD17peter,MD18parsa,MD18DT}.

\item Is it true the number of simple modular categories for a fixed rank has a polynomial growth as the rank goes to infinity?  Note that this is false for all modular categories \cite{RF16}.

\item Can we define a zeta function for all modular categories using simple modular categories as analogues of primes?

\end{enumerate}
\end{opq}

\subsection{The tetra-category of modular categories}\label{7.3}

We can imagine to use all modular categories to define a state sum invariant of $4$-manifolds $X$.  Fixing a generic celluation $X$ into $4$-cells, we color each $4$-cell by a modular category and the $3$-face between two of them by the fusion category resulting from condensing an algebra or the $T$-crossed extension; $2$-simplices to be colored by bi-modules etc.  What resulted is essentially a promotion of the tri-category of all fusion categories to a tetra-category, which does not yet have an algebraic definition.  Therefore, we will not attempt to make the idea rigorous in any way.

Heuristically, the tetra-category has: 

\begin{enumerate}

\item 0-morphism or object: modular categories

\item 1-morphism: condensable algebras or fusion categories

\item 2-morphism:  bimodules between fusion categories

\item 3-morphism: module functors between bimodules

\item 4-morphism: natural transformations between module functors

\end{enumerate}

Then a structure theory of modular categories is the same as an explicit description of this tetra-category.

\vspace{1cm}
\noindent\textbf{Acknowledgements.} \quad S.X.C. acknowledges the support from the Simons Foundation. Z.W. is partially supported by NSF grants DMS-1411212 and  FRG-1664351.

\section{Appendices}\label{8}

\subsection{Module categories over $\Vec_{\mbbZ_2}$}\label{8.1}
We want to calculate the monoidal category given by the bimodules of $\Vec{\mathbb{Z}_2}$. First, let us review the background material of \textbf{Proposition 3.19} of \cite{etingof2009fusion} described in section 3.6, 2.7. 

For $G$ a general finite group, the indecomposable semisimple module category of $\Vec_{G}$ are labelled by $\mathcal{M}(H,\psi)$ where $H$ is a subgroup and $\psi\in H^2(H,k^\times)$. There is an identification by conjugation of element of $G$. For abelian groups which have no nontrivial conjugation, there is no identification. Recall that the rank of the module category corresponding to the pair $(H,\psi)$ is equal to the number of right cosets. 

From now on, groups like $E$ are assumed to be abelian and therefore, 2-cocycles are the same as skew-symmetric bicharacters, i.e. biadditive maps $b : E \times E \to k^\times$ where $b(x,x)=1$. We shall assume $k=\mathbb{C}$.

For $N \subset E$ a subgroup and $\psi$ a bicharacter on $E$, denote the orthogonal complement of $N$ with respect to $\psi$ by $N^\perp$ to be the subgroup $\{n\in E| \psi(x,n)=1,\  \forall x \in N\}$. Notice if $\psi$ is nondegenerate then $N^\perp$ can be identified with $E/N$ and $|N|.|N^\perp|=|E|$. Also, the mere existence of a nondegenerate bicharacter implies $|E|$ being a square. In our case of study, most of the time $\psi=1$ (as $H^2$ is trivial) and therefore $N^\perp=E$.

We will need to sometimes consider how a bicharacter can be pushed forward via a map. Consider abelian groups $A,B$ with group homomorphism $\phi : B \to A$, and skew symmetric bicharacter $\xi$ on $B$. Take $K=\ker \phi$ and $K^\perp$ the orthogonal complement with respect to $\xi$. Then it can be shown that $\xi$ can be (in the most obvious way) pushed forward to define a skew-symmetric bicharacter on $H=\phi(K^\perp)$.

\textbf{Tensor product of bimodule categories.} Let $A_1,A_2,A_3$ be finite abelian groups (for example $\mathbb{Z}_2$). We would like to take the product of a $\Vec_{A_1}-\Vec_{A_2}$ module with a $\Vec_{A_2}-\Vec_{A_3}$ module and compute what the $\Vec_{A_1}-\Vec_{A_3}$ bimodule is. First, notice that a $\Vec_{A}-\Vec_{B}$ bimodule is simply an $A \oplus B^{opp}-$module but $B$ is abelian so an $A\oplus B-$module. As semisimple indecomposable modules of $\Vec_{G}$ have been classified, we have to compute
$$\mathcal{M}(H,\psi) \boxtimes_{\Vec_{A_2}} \mathcal{M}(H',\psi')$$
for $H \subset A_1\oplus A_2, H' \subset A_2 \oplus A_3$ and two skew-symmetric bicharacters $\psi,\psi'$ on $H,H'$. 

First define $H \circ  H'$ to be subgroup of elements $(a_1,-a_2,a_2,a_3)$ in $H \oplus H'$. Let $H \cap H' \subset A_2$ be their intersection inside $A_2$, i.e. $(H \cap A_2) \cap (H' \cap A_2)$; for example if $A_i=\mathbb{Z}_2$ and $H=\Delta$ the diagonal then $H \cap H'=\{0\}$, the trivial group as $\Delta \cap A_2=\{0\}$. Next, consider $H \cap H'$ as a subgroup of $H \circ H'$ by the anti-diagonal embedding $x \to (0,-x,x,0)$. Notice the bicharacter $\psi \times \psi'$ on $H \oplus H'$ can be restricted to $H \circ H'$. Hence, we can take the orthogonal complement $(H \cap H')^\perp$ of $H \cap H'$. Finally, as $H \cap H'$ is living inside the kernel of the projection $\phi: A_1 \oplus A_2 \oplus A_2 \oplus A_3 \to A_1\oplus A_3$, we can push-forward its bicharacter and get the subgroup $H''=\phi((H \cap H')^\perp) \subset A_1\oplus A_3$, with the bicharacter $\psi''=\phi_*(\psi\times\psi')$. We can finally state the \textbf{Proposition 3.19} in \cite{etingof2009fusion}:
\begin{prop}\label{prop3}
$$\mathcal{M}(H,\psi) \boxtimes_{\Vec_{A_2}} \mathcal{M}(H',\psi')=m.\mathcal{M}(H'',\psi'')$$
where 
$$m=\frac{|H \cap H'|.|(H \cap H')^\perp|.|A_2|}{|H|.|H'|}.$$
\end{prop}
Let us study the case of $\Vec_{\mathbb{Z}_2}$ bimodules when $A=A_i=\Vec_{\mathbb{Z}_2}$. There are six pairs $(H,\psi)$ for subgroups of $A \oplus A$. They give six bimodules $M_{i,j}$ where $i$ will denote the rank (which is $2|A|/|H|=4/|H|$) and $j$ will be present if there are more than one instance of rank $i$:
\begin{enumerate}
    \item $M_2=\mathcal{M}(\Delta,1)$ for the diagonal subgroup,
    \item $M_{1,1}=\mathcal{M}(A \oplus A,1)$ with the trivial bicharacter,
    \item $M_{1,2}=\mathcal{M}(A \oplus A,\xi)$ with the nontrivial bicharacter,
    \item $M_{4}=\mathcal{M}(1,1)$ with the trivial subgroup,
    \item $M_{2,1}=\mathcal{M}(A \oplus \{0\},1)$,
    \item $M_{2,2}=\mathcal{M}(\{0\} \oplus A,1)$.
\end{enumerate}
$\xi$ can be computed. Notice $2x=0, \forall x \in A \oplus A$ and $\xi(x,x)=1$ (skew-symmetry), hence
\begin{align}
\xi(x,y)\xi(x,0)=\xi(x,y+0) \implies \xi(x,0)=1, \text{similarly } \xi(0,x)=1 
\end{align}
\begin{align}
\xi(x,y)^2=\xi(x,2y)=\xi(x,0)=1 \implies \xi(x,y)= \pm 1
\end{align}
\begin{align}
\xi(x,y)=\xi(x,y)\xi(y,y)=\xi(x+y,y) \implies \xi(x,y)=\xi(x,x+y)=\xi(x+y,y)
\end{align}
\begin{align}
\xi(x,y)=\xi(x+y,y)=\xi(x+y,x+2y=x)=\xi(x+y,x)=\xi(y,x) \implies \xi \text{ symmetric} 
\end{align}
There are 16 elements in $A \oplus A$, but from the above, one needs to only determine $\xi$ for the three possible pairs from the three elements $a=(1,1),b=(1,0),c=(0,1) \in A$. But $\xi(a,b)=\xi(a,a+b)=\xi(a,c)=\xi(a+c,c)=\xi(b,c)=t$. So all their values are equal. Therefore, as $t$ can be only $\pm 1$, for $\xi$ to be nontrivial, we choose $t=-1$. This can be checked to satisfy all relations of a skew symmetric bicharacter.

From here, finding the multiplication table of the bimodule category is only straightforward calculation using what was described before the proposition. It is especially easy to do as in $\mathbb{Z}_2$, we always have $a=-a$ and the negative signs can be ignored. Further, for the $M_{i,j}$ with trivial character which gives orthogonal complement $N \subset E$ always equal to the whole group $E$, $(H'',\psi'')$ in the proposition turns out to be simply $(\phi(H \circ H'),1)$, that is the subgroup $(a_1,a_3)$ for which there exists $a_2$ such that $(a_1,a_2) \in H, (a_2,a_3) \in H'$. Notice the construction of $H \circ H'$ already shows the category is not symmetric. $M_2$ is the unit as it can be observed from the \hyperref[table1]{Table 1}. The category can also be seen to be not rigid therefore certainly not a fusion category. This could have been guessed from proposition where tensor of two bimodules is a single one and if it is a fusion category then every element is invertible, the category is pointed and therefore a group and there should be no multiplicities at all. This is very unlikely and probably not true in most cases. In brief, the category is a semisimple tensor category.

\begin{table}\label{table1}
\caption{$A\boxtimes_{\Vec{\mathbb{Z}_2}} B$ where $A$ is from left column}
\label{tab:DS3-fusion}
\begin{tabular}{|c|c|c|c|c|c|c|}
\hline $\boxtimes_{\Vec{\mathbb{Z}_2}}$ &$M_2$ &$M_{1,1}$ &$M_{1,2}$ &$M_4$ &$M_{2,1}$ &$M_{2,2}$\\ \hline
$M_2$ &$M_2$&$M_{1,1}$ &$M_{1,2}$ &$M_4$ &$M_{2,1}$ &$M_{2,2}$\\ \hline
$M_{1,1}$ &$M_{1,1}$ & $2M_{1,1}$ &$M_{2,1}$& $      M_{2,1}$ &$2M_{2,1}$ &$M_{1,1}$\\ \hline
$M_{1,2}$ &$M_{1,2}$ &$M_{2,2}$ &$M_2$& $M_{2,1}$ &$M_4$ & $M_{1,1}$\\ \hline
$M_4$ &$M_4$ &$M_{2,2}$ &$M_{2,2}$ & $2M_4$ & $M_4$ & $2M_{2,2}$ \\ \hline
$M_{2,1}$ &$M_{2,1}$ &$M_{1,1}$ &$M_{1,1}$ & $2M_{2,1}$ & $M_{2,1}$ & $2M_{2,2}$   \\ \hline
$M_{2,2}$ &$M_{2,2}$ &$2M_{2,2}$ &$M_{4}$ & $M_{4}$ & $2M_{4}$ & $M_{2,2}$\\  \hline
\end{tabular}
\end{table}
~
\\
\textit{$D(\mbbZ_2)$ bimodules.} As promised in \hyperref[6.4]{6.4}, we would like to detail the computations used to compute the tensor product of the select bimodules of $D(\mbbZ_2)$.
 
$M_{2,1}$ and $M_{2,2}$ corresponding to $A \times 0$ and $0 \times A$ are actually the $D(\mbbZ_2)$-modules which were made into bimodules using the braiding structure. Therefore, let us calculate the bimodule products $M_{2,i}^2$. We start with $i=2$. Recall that the bicharacter is always 
$$\psi((x_1,y_1),(x_2,y_2))=e^{\frac{2\pi i}{2N}(y_{22}(x_{11}+y_{11})-(x_{21}+y_{21})y_{12})}.$$
The algebra $B$ is $x\oplus y \in \mathbb{Z}_2^2 \oplus \mathbb{Z}_2^2$ such that $x+y \in 0 \times \mathbb{Z}_2$. This means it has $16$ elements
\begin{align}
(0\times \mathbb{Z}_2,0\times \mathbb{Z}_2) \cup (1 \times \mathbb{Z}_2,1 \times \mathbb{Z}_2).
\end{align}
On $B$, the bicharacter is trivial as $x_{11}+y_{11}=x_{21}+y_{21}=0$. To compute $\mathcal{M}(B,\psi_{trivial})  \boxtimes_{\Vec_{\mathbb{Z}_2^2}} \mathcal{M}(B,\psi_{trivial})$, we look at the intersection of $B \cap B$ inside the second component of $\mathbb{Z}_2^2 \oplus \mathbb{Z}_2^2$ and it is $0 \times \mathbb{Z}_2$. The number of elements inside $B \circ B$ is exactly $16$, with elements given by 
$$(0\times \mathbb{Z}_2,a,a,0\times \mathbb{Z}_2) \cup (1 \times \mathbb{Z}_2,b,b,1 \times \mathbb{Z}_2), \ where \ a\in 0\times  \mathbb{Z}_2, b \in 1 \times \mathbb{Z}_2.$$
This is equal to $(B \cap B)^\perp$ as the character is trivial. Hence, the multiplicity formula from the proposition gives
\begin{align}
m=\frac{2 . 16 . 4}{8 . 8}=2
\end{align}
and the algebra $H''$ is the image of $(B \cap B)^\perp=B \circ B$ in $\mathbb{Z}_2^2 \oplus \mathbb{Z}_2^2$ when we forget the middle two elements, which means we get $B$ back. Hence, 
$$M_{2,2} \boxtimes_{D\mathbb{Z}_2} M_{2,2}=2M_{2,2}.$$

The procedure is quite similar for the case $M_{2,1}$, where the only subtlety is when we want to compute the orthogonal complement $(B \cap B)^\perp$ as the bicharacter $\psi$ is no longer trivial. We have $B \cap B=\{ (0,a,a,0), a \in \mathbb{Z}_2\times 0\}$ and we have to find all $(x,y,y,t) \in B \circ B$ such that 
$$\psi((0,a),(x,y))\psi((a,0),(y,t))=1$$
for any element in $B \cap B$. Recall that 
$$B \circ B=\{(x,y,y,t) | x+y, y+t \in A=\mathbb{Z}_2 \times 0\}.$$
It can be seen easily that for $a=(1,0)$ and any element $(x,y,y,t)$ with $x,y,t \in \mathbb{Z}_2 \times 1$ the above will be $-1$. In fact,
$$\psi((0,a),(x,y))\psi((a,0),(y,t))=e^{\frac{\pi i}{2}(y_2+t_2)a_1}$$
and otherwise, when $x,y,t \in \mathbb{Z}_2 \times 0$, it is equal to one. Therefore the orthogonal complement has $8$ elements which are $x,y,t \in \mathbb{Z}_2 \times 0$. Hence the multiplicity can be seen to be $1$. The algebra of the result $\mathcal{M}(H,\xi)=M_{2,1} \boxtimes_{D\mathbb{Z}_2}M_{2,1}$ will be the image of the orthogonal complement as we drop the middle two components which is $H=(\mathbb{Z}_2 \times 0,\mathbb{Z}_2 \times 0)$. As for its bicharacter $\xi$, it must be the image of $\psi \times \psi$ under the projection, and the way to compute it is to take any preimage and compute $\psi \times \psi$. It can be seen to be trivial by taking preimages $(x,0,0,t),(x',0,0,t') \in (B \cap B)^\perp$ where $(x,t),(x',t') \in H \implies x,t,x',t' \in \mathbb{Z}_2 \times 0$. We compute $\psi((x,0),(x',0))\psi((0,t),(0,t'))$ and both are $1$. All in all 
\begin{align}
M_{2,2} \boxtimes_{D\mathbb{Z}_2} M_{2,2}&=2M_{2,2} \\
M_{2,1} \boxtimes_{D\mathbb{Z}_2} M_{2,1}&=\mathcal{M}((\mathbb{Z}_2 \times 0,\mathbb{Z}_2 \times 0),\psi_{trivial}).
\end{align}

\subsection{HM on $2+\tau$}\label{8.2}
We present a Hopf algebra structure on $2+\tau$ in the Fibonacci category $\Fib$. 

Throughout the section, $\golden = \frac{\sqrt{5}+1}{2}$ is the golden ratio. Some other notations are in order. For any simple objects $a,b,c$ such that $c$ is a subobject of $a \otimes b$, we arbitrarily choose and fix $V_{c}^{ab} \in \Hom(c, a \otimes b), \ V_{ab}^c \in \Hom(a \otimes b, c)$ such that $V_{ab}^c \circ V_{c}^{ab} = id_c$. The graphical representations for $V_{c}^{ab}$, $V_{ab}^c$, and $id_c$ are given in Figure \ref{fig:Vabc}.
\begin{figure}
\centering
\begin{tikzpicture}[baseline={([yshift=-.5ex]current bounding box.center)},scale=\scale, every node/.style={scale=\scale}]
\begin{scope}
\draw (0,0) -- (-1,1)node[above]{$a$};
\draw (0,0) -- (1,1)node[above]{$b$};
\draw (0,0) -- (0,-1)node[below]{$c$};
\end{scope}

\begin{scope}[xshift = 5cm]
\draw (0,0) -- (-1,-1)node[below]{$a$};
\draw (0,0) -- (1,-1)node[below]{$b$};
\draw (0,0) -- (0,1)node[above]{$c$};
\end{scope}

\begin{scope}[xshift = 9cm]
\draw (0,-1)node[below]{$c$} -- (0,1)node[above]{$c$};
\end{scope}
\end{tikzpicture}
\caption{Left: $V_{c}^{ab}$; Middle: $V_{ab}^c$; Right: $id_c$.}\label{fig:Vabc}
\end{figure}
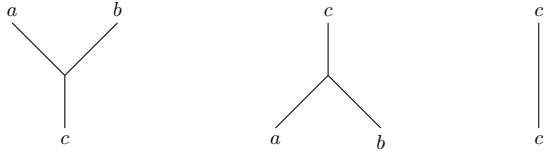
For an object $X = \sum_{a}X_a \, a$ where $X_a$ represents the multiplicities of $a$,  denote by $a_i$ the $i$-th copy of $a$ in $X$, $i = 1, \cdots, X_a$. If $X_a = 1$, we simple write $a$ instead of $a_1$. Then a morphism $f$ from $X \otimes Y$ to $Z$ is fully specified by   
\begin{align*}
f &= \sum_{a,b,c} \sum_{i,j,k}f_{abc}^{ijk} 
\begin{tikzpicture}[baseline={([yshift=-.5ex]current bounding box.center)},scale=\scale, every node/.style={scale=\scale}]
\draw (0,0) -- (-1,-1)node[below]{$a_i$};
\draw (0,0) -- (1,-1)node[below]{$b_j$};
\draw (0,0) -- (0,1)node[above]{$c_k$};
\end{tikzpicture}
\end{align*}
for some scalars $f_{abc}^{ijk}$. Usually, the terms with vanishing $f_{abc}^{ijk}$ will be dropped out from the above equation. Morphisms from $Z$ to $X \otimes Y$, or from $X$ to $Y$ are specified in a similar way in terms of $V_{c}^{ab}\,'$s or $id_c\,'$s.

With assistance of Mathematica, we conclude that there is a unique Hopf algebra structure on $2+\tau$. Here is a brief summary. The Hopf algebra is not commutative, nor cocommutative. It has a two sided integral on which the value of the counit  is not zero. Hence it is semisimple. However, the antipode has order $10$. Thus, semisimplicity of Hopf algebras in $\Fib$ is not equivalent to the condition that the antipode has order two, unlike complex finite dimensional Hopf algebras in $\text{Vec}_{\mathbb{C}}$. $2+\tau$ has four irreducible modules whose underlying objects are given by $\unit,\  \tau,\  \tau$ and  $\unit \oplus \tau$. These modules form the category $\DFib$.  we present the Hopf algebra structure maps as well as the module structures below.

The structure maps $M, \ \Delta, \ \eta, \ \epsilon, \ S$ on $2+\tau$ are defined, respectively, in Equations \ref{equ:fib_M}, \ref{equ:fib_delta}, \ref{equ:fib_eta}, \ref{equ:fib_epsilon}, \ref{equ:fib_S}, where in Equation \ref{equ:fib_delta}, we define 
\begin{align*}
V = \left(-\frac{\golden^{-2}}{2} + \frac{5^{\frac{1}{4}}\golden^{-\frac{1}{2}}}{2}i\right), \qquad \Un = \frac{i \sqrt[4]{5 \left(-9+4 \sqrt{5}-i \sqrt{1525-682 \sqrt{5}}\right)}}{2^{3/4}}
\end{align*}

\begin{align}
\label{equ:fib_M}
\begin{tikzpicture}[baseline={([yshift=-.5ex]current bounding box.center)},scale=\scale, every node/.style={scale=\scale}]
\draw (0,0) -- (-1,-1)node[below]{$\mathbf{1}_{1}$};
\draw (0,0) -- (1,-1)node[below]{$\mathbf{1}_{1}$};
\draw (0,0) -- (0,1)node[above]{$\mathbf{1}_{1}$};
\end{tikzpicture}
+ 
&\begin{tikzpicture}[baseline={([yshift=-.5ex]current bounding box.center)},scale=\scale, every node/.style={scale=\scale}]
\draw (0,0) -- (-1,-1)node[below]{$\mathbf{1}_{1}$};
\draw (0,0) -- (1,-1)node[below]{$\mathbf{1}_{2}$};
\draw (0,0) -- (0,1)node[above]{$\mathbf{1}_{2}$};
\end{tikzpicture}
+ 
&\begin{tikzpicture}[baseline={([yshift=-.5ex]current bounding box.center)},scale=\scale, every node/.style={scale=\scale}]
\draw (0,0) -- (-1,-1)node[below]{$\mathbf{1}_{2}$};
\draw (0,0) -- (1,-1)node[below]{$\mathbf{1}_{1}$};
\draw (0,0) -- (0,1)node[above]{$\mathbf{1}_{2}$};
\end{tikzpicture}
+ 
&\begin{tikzpicture}[baseline={([yshift=-.5ex]current bounding box.center)},scale=\scale, every node/.style={scale=\scale}]
\draw (0,0) -- (-1,-1)node[below]{$\mathbf{1}_{2}$};
\draw (0,0) -- (1,-1)node[below]{$\mathbf{1}_{2}$};
\draw (0,0) -- (0,1)node[above]{$\mathbf{1}_{1}$};
\end{tikzpicture}
+ 
&\begin{tikzpicture}[baseline={([yshift=-.5ex]current bounding box.center)},scale=\scale, every node/.style={scale=\scale}]
\draw (0,0) -- (-1,-1)node[below]{$\tau$};
\draw (0,0) -- (1,-1)node[below]{$\tau$};
\draw (0,0) -- (0,1)node[above]{$\mathbf{1}_{1}$};
\end{tikzpicture}
- \nonumber \\
\begin{tikzpicture}[baseline={([yshift=-.5ex]current bounding box.center)},scale=\scale, every node/.style={scale=\scale}]
\draw (0,0) -- (-1,-1)node[below]{$\tau$};
\draw (0,0) -- (1,-1)node[below]{$\tau$};
\draw (0,0) -- (0,1)node[above]{$\mathbf{1}_{2}$};
\end{tikzpicture}
+ 
&\begin{tikzpicture}[baseline={([yshift=-.5ex]current bounding box.center)},scale=\scale, every node/.style={scale=\scale}]
\draw (0,0) -- (-1,-1)node[below]{$\mathbf{1}_{1}$};
\draw (0,0) -- (1,-1)node[below]{$\tau$};
\draw (0,0) -- (0,1)node[above]{$\tau$};
\end{tikzpicture}
 -
&\begin{tikzpicture}[baseline={([yshift=-.5ex]current bounding box.center)},scale=\scale, every node/.style={scale=\scale}]
\draw (0,0) -- (-1,-1)node[below]{$\mathbf{1}_{2}$};
\draw (0,0) -- (1,-1)node[below]{$\tau$};
\draw (0,0) -- (0,1)node[above]{$\tau$};
\end{tikzpicture}
+ 
&\begin{tikzpicture}[baseline={([yshift=-.5ex]current bounding box.center)},scale=\scale, every node/.style={scale=\scale}]
\draw (0,0) -- (-1,-1)node[below]{$\tau$};
\draw (0,0) -- (1,-1)node[below]{$\mathbf{1}_{1}$};
\draw (0,0) -- (0,1)node[above]{$\tau$};
\end{tikzpicture}
-
&\begin{tikzpicture}[baseline={([yshift=-.5ex]current bounding box.center)},scale=\scale, every node/.style={scale=\scale}]
\draw (0,0) -- (-1,-1)node[below]{$\tau$};
\draw (0,0) -- (1,-1)node[below]{$\mathbf{1}_{2}$};
\draw (0,0) -- (0,1)node[above]{$\tau$};
\end{tikzpicture}
- \\
\sqrt{2}\golden^{-\frac{3}{4}}
&\begin{tikzpicture}[baseline={([yshift=-.5ex]current bounding box.center)},scale=\scale, every node/.style={scale=\scale}]
\draw (0,0) -- (-1,-1)node[below]{$\tau$};
\draw (0,0) -- (1,-1)node[below]{$\tau$};
\draw (0,0) -- (0,1)node[above]{$\tau$};
\end{tikzpicture}
&&& \nonumber
\end{align}

\begin{align}
\label{equ:fib_delta}
&\begin{tikzpicture}[baseline={([yshift=-.5ex]current bounding box.center)},scale=\scale, every node/.style={scale=\scale}]
\draw (0,0) -- (-1,1)node[above]{$\mathbf{1}_{1}$};
\draw (0,0) -- (1,1)node[above]{$\mathbf{1}_{1}$};
\draw (0,0) -- (0,-1)node[below]{$\mathbf{1}_{1}$};
\end{tikzpicture}
- \frac{1}{2\golden}
&\begin{tikzpicture}[baseline={([yshift=-.5ex]current bounding box.center)},scale=\scale, every node/.style={scale=\scale}]
\draw (0,0) -- (-1,1)node[above]{$\mathbf{1}_{1}$};
\draw (0,0) -- (1,1)node[above]{$\mathbf{1}_{1}$};
\draw (0,0) -- (0,-1)node[below]{$\mathbf{1}_{2}$};
\end{tikzpicture}
+ \frac{1}{2\golden}
&\begin{tikzpicture}[baseline={([yshift=-.5ex]current bounding box.center)},scale=\scale, every node/.style={scale=\scale}]
\draw (0,0) -- (-1,1)node[above]{$\mathbf{1}_{1}$};
\draw (0,0) -- (1,1)node[above]{$\mathbf{1}_{2}$};
\draw (0,0) -- (0,-1)node[below]{$\mathbf{1}_{2}$};
\end{tikzpicture}
+ \frac{1}{2\golden}
&\begin{tikzpicture}[baseline={([yshift=-.5ex]current bounding box.center)},scale=\scale, every node/.style={scale=\scale}]
\draw (0,0) -- (-1,1)node[above]{$\mathbf{1}_{2}$};
\draw (0,0) -- (1,1)node[above]{$\mathbf{1}_{1}$};
\draw (0,0) -- (0,-1)node[below]{$\mathbf{1}_{2}$};
\end{tikzpicture}
+\nonumber\\
\frac{\sqrt{5}}{2\golden}
&\begin{tikzpicture}[baseline={([yshift=-.5ex]current bounding box.center)},scale=\scale, every node/.style={scale=\scale}]
\draw (0,0) -- (-1,1)node[above]{$\mathbf{1}_{2}$};
\draw (0,0) -- (1,1)node[above]{$\mathbf{1}_{2}$};
\draw (0,0) -- (0,-1)node[below]{$\mathbf{1}_{2}$};
\end{tikzpicture}
+ 
V
&\begin{tikzpicture}[baseline={([yshift=-.5ex]current bounding box.center)},scale=\scale, every node/.style={scale=\scale}]
\draw (0,0) -- (-1,1)node[above]{$\tau$};
\draw (0,0) -- (1,1)node[above]{$\tau$};
\draw (0,0) -- (0,-1)node[below]{$\mathbf{1}_{2}$};
\end{tikzpicture}
+ \frac{1}{2\golden}
&\begin{tikzpicture}[baseline={([yshift=-.5ex]current bounding box.center)},scale=\scale, every node/.style={scale=\scale}]
\draw (0,0) -- (-1,1)node[above]{$\mathbf{1}_{1}$};
\draw (0,0) -- (1,1)node[above]{$\tau$};
\draw (0,0) -- (0,-1)node[below]{$\tau$};
\end{tikzpicture}
+ \frac{\sqrt{5}}{2\golden}
&\begin{tikzpicture}[baseline={([yshift=-.5ex]current bounding box.center)},scale=\scale, every node/.style={scale=\scale}]
\draw (0,0) -- (-1,1)node[above]{$\mathbf{1}_{2}$};
\draw (0,0) -- (1,1)node[above]{$\tau$};
\draw (0,0) -- (0,-1)node[below]{$\tau$};
\end{tikzpicture}
+  \\
 \frac{1}{2\golden}
&\begin{tikzpicture}[baseline={([yshift=-.5ex]current bounding box.center)},scale=\scale, every node/.style={scale=\scale}]
\draw (0,0) -- (-1,1)node[above]{$\tau$};
\draw (0,0) -- (1,1)node[above]{$\mathbf{1}_{1}$};
\draw (0,0) -- (0,-1)node[below]{$\tau$};
\end{tikzpicture}
+
\frac{\sqrt{5}}{2\golden}
&\begin{tikzpicture}[baseline={([yshift=-.5ex]current bounding box.center)},scale=\scale, every node/.style={scale=\scale}]
\draw (0,0) -- (-1,1)node[above]{$\tau$};
\draw (0,0) -- (1,1)node[above]{$\mathbf{1}_{2}$};
\draw (0,0) -- (0,-1)node[below]{$\tau$};
\end{tikzpicture}
-\Un
&\begin{tikzpicture}[baseline={([yshift=-.5ex]current bounding box.center)},scale=\scale, every node/.style={scale=\scale}]
\draw (0,0) -- (-1,1)node[above]{$\tau$};
\draw (0,0) -- (1,1)node[above]{$\tau$};
\draw (0,0) -- (0,-1)node[below]{$\tau$};
\end{tikzpicture}
& \nonumber
\end{align}

\begin{align}
\label{equ:fib_eta}
\begin{tikzpicture}[baseline={([yshift=-.5ex]current bounding box.center)},scale=\scale, every node/.style={scale=\scale}]
\draw (0,0) node[below] {$\mathbf{1}$} -- (0,2) node[above]{$\mathbf{1}_{1}$};
\end{tikzpicture}
\end{align}

\begin{align}
\label{equ:fib_epsilon}
\begin{tikzpicture}[baseline={([yshift=-.5ex]current bounding box.center)},scale=\scale, every node/.style={scale=\scale}]
\draw (0,0) node[below] {$\mathbf{1}_{1}$} -- (0,2) node[above]{$\mathbf{1}$};
\end{tikzpicture}
\ +\  
\begin{tikzpicture}[baseline={([yshift=-.5ex]current bounding box.center)},scale=\scale, every node/.style={scale=\scale}]
\draw (0,0) node[below] {$\mathbf{1}_{2}$} -- (0,2) node[above]{$\mathbf{1}$};
\end{tikzpicture}
\end{align}

\begin{align}
\label{equ:fib_S}
\begin{tikzpicture}[baseline={([yshift=-.5ex]current bounding box.center)},scale=\scale, every node/.style={scale=\scale}]
\draw (0,0) node[below] {$\mathbf{1}_{1}$} -- (0,2) node[above]{$\mathbf{1}_{1}$};
\end{tikzpicture}
\ + \ 
\begin{tikzpicture}[baseline={([yshift=-.5ex]current bounding box.center)},scale=\scale, every node/.style={scale=\scale}]
\draw (0,0) node[below] {$\mathbf{1}_{2}$} -- (0,2) node[above]{$\mathbf{1}_{2}$};
\end{tikzpicture}
\ -\  \left(\frac{\golden^{-1}}{2} + \frac{5^{\frac{1}{4}}\golden^{\frac{1}{2}}}{2}i\right)
\begin{tikzpicture}[baseline={([yshift=-.5ex]current bounding box.center)},scale=\scale, every node/.style={scale=\scale}]
\draw (0,0) node[below] {$\tau$} -- (0,2) node[above]{$\tau$};
\end{tikzpicture}
\end{align}

$2+\tau$ has four equivalent classes of irreducible modules. The trivial module is given by $M_0 = (\unit, \epsilon)$. There are two module structures on $\tau$ which we denote by $M_i = (\tau, r_i)$, $i=1,2$. There is also a module on $\unit + \tau$ denoted by $M_3 = (\unit + \tau, r_3)$. The action $r_i$ giving the module structure for $i=1,2,3$ is presented, respectively, in Equations \ref{equ:fib_mod1}, \ref{equ:fib_mod2}, \ref{equ:fib_mod3}. With the data of $2+\tau$ and its irreducible modules, it is straightforward to check (in Mathematica) that the category of $(2 + \tau)$-modules is isomorphic to $\DFib$.

\begin{align}
\label{equ:fib_mod1}
\begin{tikzpicture}[baseline={([yshift=-.5ex]current bounding box.center)},scale=\scale, every node/.style={scale=\scale}]
\draw (0,0) -- (-1,-1)node[below]{$\mathbf{1}_{1}$};
\draw (0,0) -- (1,-1)node[below]{$\tau$};
\draw (0,0) -- (0,1)node[above]{$\tau$};
\end{tikzpicture}
\ +\  
\begin{tikzpicture}[baseline={([yshift=-.5ex]current bounding box.center)},scale=\scale, every node/.style={scale=\scale}]
\draw (0,0) -- (-1,-1)node[below]{$\mathbf{1}_{2}$};
\draw (0,0) -- (1,-1)node[below]{$\tau$};
\draw (0,0) -- (0,1)node[above]{$\tau$};
\end{tikzpicture}
\end{align}

\begin{align}
\label{equ:fib_mod2}
\begin{tikzpicture}[baseline={([yshift=-.5ex]current bounding box.center)},scale=\scale, every node/.style={scale=\scale}]
\draw (0,0) -- (-1,-1)node[below]{$\mathbf{1}_{1}$};
\draw (0,0) -- (1,-1)node[below]{$\tau$};
\draw (0,0) -- (0,1)node[above]{$\tau$};
\end{tikzpicture}
\ -\ 
\begin{tikzpicture}[baseline={([yshift=-.5ex]current bounding box.center)},scale=\scale, every node/.style={scale=\scale}]
\draw (0,0) -- (-1,-1)node[below]{$\mathbf{1}_{2}$};
\draw (0,0) -- (1,-1)node[below]{$\tau$};
\draw (0,0) -- (0,1)node[above]{$\tau$};
\end{tikzpicture}
\ +\  \sqrt{2}\golden^{\frac{1}{4}}
\begin{tikzpicture}[baseline={([yshift=-.5ex]current bounding box.center)},scale=\scale, every node/.style={scale=\scale}]
\draw (0,0) -- (-1,-1)node[below]{$\tau$};
\draw (0,0) -- (1,-1)node[below]{$\tau$};
\draw (0,0) -- (0,1)node[above]{$\tau$};
\end{tikzpicture}
\end{align}

\begin{align}
\label{equ:fib_mod3}
\begin{tikzpicture}[baseline={([yshift=-.5ex]current bounding box.center)},scale=\scale, every node/.style={scale=\scale}]
\draw (0,0) -- (-1,-1)node[below]{$\mathbf{1}_{1}$};
\draw (0,0) -- (1,-1)node[below]{$\mathbf{1}$};
\draw (0,0) -- (0,1)node[above]{$\mathbf{1}$};
\end{tikzpicture}
-
&\begin{tikzpicture}[baseline={([yshift=-.5ex]current bounding box.center)},scale=\scale, every node/.style={scale=\scale}]
\draw (0,0) -- (-1,-1)node[below]{$\mathbf{1}_{2}$};
\draw (0,0) -- (1,-1)node[below]{$\mathbf{1}$};
\draw (0,0) -- (0,1)node[above]{$\mathbf{1}$};
\end{tikzpicture}
+ 
&\begin{tikzpicture}[baseline={([yshift=-.5ex]current bounding box.center)},scale=\scale, every node/.style={scale=\scale}]
\draw (0,0) -- (-1,-1)node[below]{$\tau$};
\draw (0,0) -- (1,-1)node[below]{$\tau$};
\draw (0,0) -- (0,1)node[above]{$\mathbf{1}$};
\end{tikzpicture}
+ 
&\begin{tikzpicture}[baseline={([yshift=-.5ex]current bounding box.center)},scale=\scale, every node/.style={scale=\scale}]
\draw (0,0) -- (-1,-1)node[below]{$\mathbf{1}_{1}$};
\draw (0,0) -- (1,-1)node[below]{$\tau$};
\draw (0,0) -- (0,1)node[above]{$\tau$};
\end{tikzpicture}
- \nonumber\\
\begin{tikzpicture}[baseline={([yshift=-.5ex]current bounding box.center)},scale=\scale, every node/.style={scale=\scale}]
\draw (0,0) -- (-1,-1)node[below]{$\mathbf{1}_{2}$};
\draw (0,0) -- (1,-1)node[below]{$\tau$};
\draw (0,0) -- (0,1)node[above]{$\tau$};
\end{tikzpicture}
+ 2
&\begin{tikzpicture}[baseline={([yshift=-.5ex]current bounding box.center)},scale=\scale, every node/.style={scale=\scale}]
\draw (0,0) -- (-1,-1)node[below]{$\tau$};
\draw (0,0) -- (1,-1)node[below]{$\mathbf{1}$};
\draw (0,0) -- (0,1)node[above]{$\tau$};
\end{tikzpicture}
 -&\sqrt[4]{4\sqrt{5}-8}
&\begin{tikzpicture}[baseline={([yshift=-.5ex]current bounding box.center)},scale=\scale, every node/.style={scale=\scale}]
\draw (0,0) -- (-1,-1)node[below]{$\tau$};
\draw (0,0) -- (1,-1)node[below]{$\tau$};
\draw (0,0) -- (0,1)node[above]{$\tau$};
\end{tikzpicture}
\end{align}

\subsection{Fusion rules of $(SO(8)_1)_{S_3}^{\times,S_3}$}\label{8.3}
Here, we list the fusion rules of example \hyperref[5.2.3]{5.2.3} as provided in \cite[Appendix B]{cui2016gauging} in our new notation; it should be noted the two theories in \cite[Appendix B]{cui2016gauging} are the same. We will use $\pm$ and $\mp$ to shorten the list. The only rule is that if a choice is made for the first (second) superscript then the corresponding choice must be made throughout the equation for the first (second) superscript; e.g. in $X^{\pm \pm} \otimes \ldots=X^{\pm \mp} \oplus \ldots$, we can choose $X^{+ -} \otimes \ldots=X^{+ +} \oplus \ldots$.
\begin{itemize}
\item $B \otimes B = A $

\item $B \otimes C = C $

\item $B \otimes Y_\pm = Y_\mp $

\item $B \otimes X = X $

\item $B \otimes {}_{\alpha}X = {}_{\alpha}X $

\item $B \otimes {}_{\alpha^*}X = {}_{\alpha^*}X $

\item $B \otimes X^{\pm\pm} = X^{\pm \mp} $

\item $C \otimes C = A  \oplus B \oplus C$

\item $C \otimes Y_\pm = Y_+  \oplus Y_-$

\item $C \otimes X = {}_{\alpha}X  \oplus {}_{\alpha^*}X$

\item $C \otimes {}_{\alpha}X = X  \oplus {}_{\alpha^*}X$

\item $C \otimes {}_{\alpha^*}X = X  \oplus {}_{\alpha}X$

\item $C \otimes X^{\pm \pm} = X^{\pm \pm}  \oplus X^{\pm \mp}$

\item $Y_\pm \otimes Y_\pm = A  \oplus C \oplus Y_+ \oplus Y_-$

\item $Y_+ \otimes Y_- = B  \oplus C \oplus Y_+ \oplus Y_-$

\item $Y_\pm \otimes Z = X  \oplus {}_{\alpha}X \oplus {}_{\alpha^*}X, \forall Z \in \{X,{}_{\alpha}X,{}_{\alpha^*}X\}$,

\item $X \otimes X = A  \oplus B \oplus Y_+ \oplus Y_- \oplus X \oplus {}_{\alpha^*}X$

\item $X \otimes {}_{\alpha}X = C  \oplus Y_+ \oplus Y_- \oplus {}_{\alpha}X \oplus {}_{\alpha^*}X$

\item $X \otimes {}_{\alpha^*}X = C  \oplus Y_+ \oplus Y_- \oplus X \oplus {}_{\alpha}X$

\item $Z \otimes X^{\pm \pm} = \oplus_{i,j\in\{\pm1\}} X^{ij}, \ \forall Z \in \{X,{}_{\alpha}X,{}_{\alpha^*}X\}$

\item ${}_{\alpha}X \otimes {}_{\alpha}X = A  \oplus B \oplus Y_+ \oplus Y_- \oplus X \oplus {}_{\alpha}X$

\item ${}_{\alpha}X \otimes {}_{\alpha^*}X = C  \oplus Y_+ \oplus Y_- \oplus X \oplus {}_{\alpha^*}X$

\item ${}_{\alpha^*}X \otimes {}_{\alpha^*}X = A  \oplus B \oplus Y_+ \oplus Y_- \oplus {}_{\alpha}X \oplus {}_{\alpha^*}X$

\item $X^{\pm \pm} \otimes X^{\mp \pm} = Y_+  \oplus Y_- \oplus X \oplus {}_{\alpha}X \oplus {}_{\alpha^*}X$, exceptionally in this equation the \textit{second} superscripts can be chosen independently

\item $Y_+ \otimes X^{++} = X^{++}  \oplus X^{-+} \oplus X^{--}$

\item $Y_+ \otimes X^{+-} = X^{+-}  \oplus X^{-+} \oplus X^{--}$

\item $Y_+ \otimes X^{-+} = X^{++}  \oplus X^{+-} \oplus X^{--}$

\item $Y_+ \otimes X^{--} = X^{++}  \oplus X^{+-} \oplus X^{-+}$

\item $Y_- \otimes X^{++} = X^{+-}  \oplus X^{-+} \oplus X^{--}$

\item $Y_- \otimes X^{+-} = X^{++}  \oplus X^{-+} \oplus X^{--}$

\item $Y_- \otimes X^{-+} = X^{++}  \oplus X^{+-} \oplus X^{-+}$

\item $Y_- \otimes X^{--} = X^{++}  \oplus X^{+-} \oplus X^{--}$

\item $X^{\pm \pm} \otimes X^{\pm \pm} = A  \oplus C \oplus Y_\pm \oplus X \oplus {}_{\alpha}X \oplus {}_{\alpha^*}X$, where the same choice is for the first subscript and first superscript

\item $X^{++} \otimes X^{+-} = B  \oplus C \oplus Y_- \oplus X \oplus {}_{\alpha}X \oplus {}_{\alpha^*}X$

\item $X^{-+} \otimes X^{--} = B  \oplus C \oplus Y_+ \oplus X \oplus {}_{\alpha}X \oplus {}_{\alpha^*}X$
\end{itemize}

\bibliographystyle{apa}
\bibliography{main}
\address{\textsuperscript{1\label{1.}}Stanford Institute for Theoretical Physics, Stanford University,
Stanford, CA 94305, U.S.A.}
\address{\textsuperscript{2\label{2.}}Dept of Mathematics, Virginia Polytechnic Institute and State University,
Blacksburg, VA 24061, U.S.A.}
\address{\textsuperscript{3\label{3.}}Dept of Mathematics, University of California,
Santa Barbara, CA 93106-6105, U.S.A.}
\address{\textsuperscript{4\label{4.}}Microsoft Station Q and Dept of Mathematics, University of California,
Santa Barbara, CA 93106-6105, U.S.A.} 

\end{document}